\documentclass[review,12]{elsarticle}
\usepackage{amsmath}
\usepackage{lineno,hyperref}
\usepackage{tikz}
\usepackage{slashbox}
\usepackage{color}
\modulolinenumbers[5]

\newtheorem{remark}{Remark}

\newtheorem{lem}{Lemma}
\newtheorem{prop}{Proposition}
\newtheorem{theo}{Theorem}
\newtheorem{proof}{Proof}
\newtheorem{Pb}{Problem}
\newtheorem{Ass}{Assumption} 
\usepackage{amssymb}

\journal{Computers \& Mathematics with Applications}
\begin{document}

\begin{frontmatter}

\title{Convergence of the Two Point Flux Approximation and a novel  fitted Two-Point Flux Approximation method for pricing options}

%\author{ Rock  S. Koffi\fnref{rock}}
%\author{Antoine  Tambue\fnref{antoine}}
%\address{African Institute for Mathematical Sciences(AIMS), 6-8 Melrose Road, Muizenberg 7945, South Africa}

\author[aims,uct]{Rock  S. Koffi}
 \ead{rock@aims.ac.za}
 \address[aims]{The African Institute for Mathematical Sciences(AIMS), 6-8 Melrose Road, Muizenberg 7945, South Africa}
 \address[uct]{Department of Mathematics and Applied Mathematics, University of Cape Town, 7701 Rondebosch, South Africa}

%\author[at,atb]{Antoine Tambue}
\author[at,atb,atc]{Antoine Tambue}
\cortext[cor1]{Corresponding author}
\ead{antonio@aims.ac.za}
\address[at]{Department of Computing Mathematics and Physics,  Western Norway University of Applied Sciences, Inndalsveien 28, 5063 Bergen.}
%\ead{antonio@aims.ac.za}
\address[atc]{The African Institute for Mathematical Sciences(AIMS),
6-8 Melrose Road, Muizenberg 7945, South Africa}
\address[atb]{Center for Research in Computational and Applied Mechanics (CERECAM), and Department of Mathematics and Applied Mathematics, University of Cape Town, 7701 Rondebosch, South Africa.}

%%
%%\fntext[antoine]{Corresponding author \\
%%			Western Norway University of Applied Sciences,  Inndalsveien 28, 5063  Bergen, Norway,\\
%%			The African Institute for Mathematical Sciences(AIMS), 6-8 Melrose Road, Muizenberg 7945, South Africa\\
%%			%Fax: +47-555-89672\\
%%			Center for Research in Computational and Applied Mechanics (CERECAM), and Department of Mathematics and Applied Mathematics, University of Cape Town, 7701 Rondebosch, South Africa.\\
%%			Tel.: +47 55 58 70 06, email:antonio@aims.ac.za, antoine.tambue@hvl.no, tambuea@gmail.com}
%%		
%%\fntext[rock]{African Institute for Mathematical Sciences(AIMS), 6-8 Melrose Road, Muizenberg 7945, South Africa\\
%%Department of Mathematics and Applied Mathematics, University of Cape Town, 7701 Rondebosch, South Africa\\
%%email:rock@aims.ac.za }


%\newtheorem{def}{Definition}

\begin{abstract}
In this paper, we  deal with numerical approximations for solving the Black-Scholes Partial Differential Equation (PDE). This PDE is well known to be degenerated.
The space discretization is performed using the classical finite volume method with Two-Point Flux Approximation (TPFA) and a novel scheme called fitted Two-Point Flux Approximation (TPFA).
 The fitted Two-Point Flux Approximation (TPFA)  combines the fitted finite volume method and the  standard TPFA method. 
 %For the spatial discretization of the Black Scholes PDE,  
More precisely the fitted finite volume method is used when the stock price approaches zero in order to handle the degeneracy of the PDE and the TPFA method 
is used on the rest of space domain.  This combination yields our  novel fitted  TPFA scheme.
 The Euler  method is used for the time discretization. We provide the rigorous convergence proofs of 
 the two  fully discretized schemes.  Numerical experiments to support  our theoretical results are provided.
\end{abstract}


\end{frontmatter}

\section{Introduction}

A financial market is a market whereby investors, companies and governments meet to trade financial securities such as  bonds, stocks,
 precious metal.  A good number of transactions on financial markets are about buying and selling options. Indeed, an option is a contract
  which gives the right to buy (call) or to sell (put) an underlying asset at an agreed price (strike) on (European options) or before (American options) a specified date (maturity). 
In their seminal paper \cite{black1973pricing}, Fisher Black and Myron Scholes stated the famous  Black Scholes model. Under some assumptions, 
the derivation of the Black Scholes model leads to a second order parabolic Partial Differential Equation (PDE) with respect to time and the underlying stock price.
 An analytical solution has been found for the Black Scholes PDE only for pricing European options with constant coefficients.
 %but it is a very challenging and complex task for other types of options. 
 Therefore, numerically algorithms are usually used to solve the general Black Scholes PDE. 
 The first numerical method used to solve the Black Scholes PDE  was the lattice method in \cite{cox1979option}. 
 Afterwards, several numerical methods such as the finite difference method \cite{duffy2013finite}, the finite volume method and the finite elements 
 methods \cite{topper2005financial} have been used to solve the Black Scholes PDE. 
 However, the Black Scholes PDE is degenerated when  the stock price approaches zero (see \cite{duffy2013finite}). This degeneracy may affect the accuracy of the numerical method used if sophistical technique is not used. 
 Thereby in \cite{wang2004novel}, S. Wang proposed a fitted finite volume method with the corresponding convergence proof in space  to tackle  the degeneracy of the Black- Scholes PDE. 
 Moreover, under less restrictive and more realistic assumptions, a convergence proof in space of the fitted finite volume method for pricing American options is proposed in \cite{wang2006power}. 
 Furthermore, in \cite{angermann2007convergence}, a rigorous convergence proof of a fully discretized scheme using the fitted finite volume method and $\theta-$Euler method for pricing both American and European options is provided. 
 In \cite{koffi2019fitted}, a novel fitted Multi-Point Flux Approximation (MPFA) to pricing two dimensional options have been provided where the standard MPFA have been coupled with the fitted technique in \cite{angermann2007convergence, wang2006power}.
 It has  also been shown computationally in \cite{koffi2019fitted} that the fitted MPFA is more accurate than the fitted finite volume method \cite{angermann2007convergence, wang2006power} for pricing options.
  In one dimensional case, the MPFA method is simply reduced to the classical Two Point Flux Approximation (TPFA) method. 
  Note that  convergence proofs of the TPFA method  are provided in \cite{eymard2000finite,tambue2016exponential}  for non degenerated parabolic PDE.  Their proofs are based 
  on the fact that the diffusion coefficient can not reach zero, therefore  such proofs can not be extended to the degenerated Black  Scholes PDE where the diffusion coefficient is zero at $s=0$ (stock price is zero).
  To the best of our knowledge, the convergence of   classical TPFA method for degenerated  PDE  have been lacking  in the literature due to the complexity of that degeneracy.
  %Moreover, A. Tambue ,in \cite{tambue2016exponential}, provided a rigorous convergence proof of a fully discretized scheme using  the TPFA method in high dimension for a non-degenerate PDE.  \newline

In this work, we fill the gap by providing a rigorous convergence proof of a fully discretized scheme using  the classical TPFA method  for degenerated Black Scholes PDE in one dimension.
Furthermore,  we also  derive the fitted-TPFA method for the  degenerated Black Scholes PDE by combining  the classical TPFA method  and the fitted finite volume method\cite{angermann2007convergence, wang2006power}
 and provide rigorous convergence proof of a fully discretized scheme where the time discretization is performed using the classical Euler methods.
% the combination, for the spatial discretization,  of the TPFA method introduced in \cite{aavatsmark2007multipoint} and the fitted finite volume \cite{wang2004novel}. 
Note the fitted finite volume method \cite{angermann2007convergence, wang2006power}  in  this combination  has for goal to handle the degeneracy of the PDE when the stock price approaches zero. 

%Our main contribution is to bound the transmissibillity coefficient obtained from the TPFA method despite the fact that the coefficient of the second order derivative of the Black Scholes PDE doesn't have a non-zero lower bound.   \newline
%\textbf{At the end}
The outline of this work is the following. In section \ref{cont-prob}, notations and mathematical setting of the continuous  problem are provided.
%5 we establish existence and uniqueness result for the continuous problem. 
 In section \ref{finfit}  the spatial discretization   using the standard finite volume method with Two Point Flux Approximation  are provided along with the corresponding novel fitted scheme.
  The coercivity proofs of the corresponding discrete bilinear  forms are also provided to ensure the existence and uniqueness of the discrete solution after TPFA  spatial discretization and fitted TPFA spatial discretization.
  The full discretization of the Black Sholes degenerated PDE and the convergence results are performed in section \ref{ful-dis}.  Note that the temporal discretization is performed using the standard $\theta-$ Euler method.
 Finally, numerical experiments are given in section \ref{num-exp} to support theoretical results.
% 
%Afterwards, we describe the TPFA method combined to the upwind method for the spatial discretization of the Black Scholes PDE in section \ref{finfit}.
% In this same section, we demonstrate the important flux consistency result. The full discretization of the PDE and the convergence results are given in section \ref{ful-dis}. 
% Finally, numerical experiment are given in section \ref{num-exp} to suppport theoritical results.

%\linenumbers
\section{Notations and mathematical setting}
\label{cont-prob}
Let us first introduce some functional spaces and their associated norm that we will  use in this work. For an open set 
$\Omega \subset \mathbb{R}$, the space of square integrable functions is denoted  $L^2(\Omega)$. We denote also by $C(\Omega)$ (respectively $C(\overline{\Omega})$) the set of continuous functions over $\Omega$ (respectively on $\overline{\Omega}$).  For any Hilbert space $G(\Omega)$ of classes of functions defined on $\Omega$, we let $L^2\big(0,T;G(\Omega)\big)$ denote the space defined by
\begin{equation}
L^2\Big((0,T);G(\Omega)\Big)=\Bigg\{  v /  v(\cdot,t)\in G(\Omega)~a.e~in~(0,T); \vert\vert v(\cdot,t)\vert\vert_{G}\in L^2\big((0,T)\big)\Bigg\},
\end{equation}
where $\vert\vert \cdot \vert\vert_{G}$ denotes the natural norm on $G(\Omega)$.
 The norm on this  space is denoted by $\vert\vert \cdot \vert\vert_{L^2\big(0,T;G(\Omega)\big)}$  and is  defined  by
\begin{equation}
\vert\vert v \vert\vert_{L^2\big(0,T;G(\Omega)\big)}=\Bigg( \int_0^T \vert\vert v(\cdot,t)\vert\vert_G^2 dt\Bigg)^{1/2}.
\end{equation}
The Black-Scholes operator being degenerated, we introduce a weighted $L^2$-norm $\vert\vert \cdot  \vert\vert_{\omega}$  defines by
\begin{equation}
\vert\vert v \vert\vert_{\omega}=\Bigg(\int_{\Omega} x^2v^2 dx\Bigg)^{1/2}.
\end{equation}
The space of all weighted square integrable functions is defined as
\begin{equation}
L^2_{\omega}(\Omega)=\Bigg\{ v: \vert\vert v \vert\vert_{\omega} < \infty \Bigg\},
\end{equation}
and the corresponding weighted inner product on $L^2_{\omega}(\Omega)$ by
\begin{equation}
 (u,v)_{\omega}= \int_{\Omega} x^2 uv dx.
\end{equation}
Thereby, we define the weighted Sobolev spaces as follows:
\begin{eqnarray}
\label{h1}
H^1_{\omega}(\Omega)=\Bigg\{v\in\mathbf{L}_2(\Omega): \exists g \in L^2_{\omega}(\Omega)~~~such~~that ~~\int_{\Omega} v \varphi' = -\int_{\Omega} g \varphi~~~~~~~\forall \varphi\in C_c(\Omega)\Bigg\}.
\end{eqnarray}
Note that in \eqref{h1}, $g=v'$ is the weak derivative. We also denote by
\begin{equation}
H^1_{0,\omega}(\Omega)=\Bigg\{v:v\in\mathbf H^1_{\omega}(\Omega)~~and~~v\big\vert_{\partial \Omega}=0\Bigg\}.
\end{equation}
Using the inner products on $L^2(\Omega)$ and $L^2_{\omega}(\Omega)$, we define   the norm $\vert\vert \cdot \vert\vert_{1,\omega}$  on $H^1_{\omega}(\Omega)$ by
%\begin{equation}
%\big(\cdot,\cdot\big)_H= \big(\cdot,\cdot\big)+\big(\cdot,\cdot\big)_{\omega} 
%\end{equation}
%with the corresponding norm denoted by $\vert\vert \cdot \vert\vert_{1,\omega}$ and defines as follows:
\begin{equation}
\big\vert\big\vert v \big\vert\big\vert_{1,\omega} = \Bigg[\big\vert\big\vert v \big\vert\big\vert^2_{L^2(\Omega)}+\big\vert\big\vert v' \big\vert\big\vert^2_{\omega}\Bigg]^{1/2}=\Bigg[(v,v)+(x^2v',v')\Bigg]^{1/2}.
\end{equation} 
 % \section{The continuous problem}
%\label{cont-prob}
Without loss the generality \footnote{The American  options are solved by just adding a nonlinear term  called Penalty}, let us consider the Black-Scholes equation PDE
\begin{equation}
\label{BS-PDE}
LV:=\frac{\partial V}{\partial t}-\frac{1}{2}\sigma^2(t)x^2\frac{\partial^2 V}{\partial x^2}-r(t)x\frac{\partial V}{\partial x}+r(t)V=0,
\end{equation}
in $(x,t)\in \Omega=(0,x_{\max}) \times (0,T]$, where $V$ is the option value, $x$ the stock price, $\sigma$ the volatility, $r$ the risk free interest ,$t$ is the time  and  $T$  is the maturity time.
%For European  put options, the corresponding initial and boundary conditions for European put options are 
%\begin{equation}
%\label{bound-cond}
%\left\lbrace \begin{array}{lcr}
%V(x,0) & =  & \max(K-x,0) \\
%V(0,t) ~~~~~~~~~~~~~~& =  K &    \\
%V(x_{\max},t) & = 0 ,& 
%\end{array} 
%\right.	
%\end{equation}
%where $K$ is the strike price.
The corresponding initial and boundary conditions are:
\begin{equation}
\label{bound-cond}
\left\lbrace \begin{array}{lcr}
V(x,0) & =  & g_1(x) \\
V(0,t) & =  & g_2(t)   \\
V(x_{\max},t) & = & g_3(x),
\end{array} 
\right.	
\end{equation}
where $g_1, g_2$ and $g_3$ are functions depending on type of options we are pricing.
Our study is conducted under the following assumption.
\begin{Ass}
\label{assum1}
	%There exists a positive constant $\beta$  such that
We assume that the coefficient $r$ and $\sigma$ are sufficiently smooth and satisfy 
\begin{eqnarray}
0\leq r(t) \leq \bar{r} ~~~~~~\underline{\sigma} \leq \sigma(t) \leq \bar{\sigma}
\end{eqnarray}
and we denote by
	\begin{equation}
	\label{ass-sig}
	\beta:= \sup_{t\in[0,T]}\sigma^2(t).
	\end{equation}
\end{Ass}
%We assume that the coefficient $r$ and $\sigma$ are sufficiently smooth and satisfy 
%\begin{eqnarray}
%0\leq r(t) \leq \bar{r} ~~~~~~\underline{\sigma} \leq \sigma(t) \leq \bar{\sigma}
%\end{eqnarray}
%We denote by $L$ the Black-Scholes operator.
% The corresponding initial and boundary conditions for European put options are 
%\begin{equation}
%\label{bound-cond}
%\left\lbrace \begin{array}{lcr}
%V(x,0) ~~~~~& =  & \max(K-x,0) \\
%V(0,t) ~~~~~~~~~~~~~~& =  K &    \\
%V(x_{\max},t) & = 0 & 
%\end{array} 
%\right.	
%\end{equation}
%where $K$ is the strike price. 
%Let us make the following assumption
%\begin{Ass}
%	There exists a positive constant $\beta$  such that
%	
%	\begin{equation}
%	\label{ass-sig}
%	\beta:= \sup_{t\in[0,T]}\sigma^2(t)
%	\end{equation}
%\end{Ass}
%
Multiplying by $e^{\beta t}$ and adding $f(x,t)=-e^{\beta t}LV_0(x)$ to both sides of (\ref{BS-PDE}), we can therefore transform the  boundary conditions in (\ref{bound-cond}) to   homogeneous Dirichlet boundary conditions 
by using the following linear transformation
%\begin{equation}
%V_0(x)=\Big(1-\frac{x}{x_{\max}}\Big)K.
%\end{equation}
\begin{equation}
V_0(x,t)=g_2(t)+\dfrac{x}{x_{\max}}\Big(g_3(t)-g_2(t)\Big).
\end{equation}
Furthermore, we introduce a new variable $u=e^{\beta t}\big(V-V_0\big)$ and we get  this new PDE  in its following divergence form:
\begin{equation}
\label{self-adj}
\frac{\partial u}{\partial t}-\frac{\partial}{\partial x}\Bigg[a(t)x^2\frac{\partial u}{\partial x}+b(t)xu\Bigg]+c(t)u=f(x,t),
\end{equation}
where
\begin{equation}
a(t)=\frac{1}{2}\sigma^2(t),~~~~b(t)=r(t)-\sigma^2(t),~~~~c(t)=2r(t)+\beta-\sigma^2(t),
\end{equation}
with the following initial and  homogeneous boundary conditions
\begin{equation}
\label{bound-cond-homo}
\left\lbrace \begin{array}{lcr}
u(0,t)=0=u(x_{\max},t) ~~~~~~~t\in [0,T)\\
% \nonumber\\
u(x,0)  =  g_1(x)-V_0(x)~~~~~~~x\in \Omega.
\end{array} 
\right.	
\end{equation}
It can be proved the solving \eqref{self-adj}-\eqref{bound-cond-homo} is equivalent to the following problem:
\begin{Pb}
	%Find a weak solution of (\ref{self-adj}) which  is a 
	Find  the function $u\in L^2(0, T,H^1_{0,\omega}(\Omega))$ such that
	\begin{equation}
	\label{weak-form}
		(u'(t),v)+A(u,v;t)=\Big(f,v\Big)~~~~\forall~v\in H^1_{0,\omega}(\Omega)
	\end{equation}	
	with
	\begin{equation}
	A(u,v;t):= \Big(ax^2u'+bxu,v'\Big)+\Big(cu,v\Big).
	\end{equation}
\end{Pb}
The follolwing lemma is key to prove the existence and uniqueness of the solution of \textbf{Problem 1}.
\begin{lem}
\label{lem1}
	There exist positive constants $C$ and $M$ such that for any $v,z\in H_{0,\omega}^1(\Omega)$,
	\begin{eqnarray}
	A(v,v;t)  & \geq  &  C \vert\vert v \vert\vert_{1,\omega}^2, \\
	A(z,v;t) &\leq &   M \cdot \vert\vert z \vert\vert_{1,\omega} \cdot  \vert\vert v \vert\vert_{1,\omega}.
	\end{eqnarray}
\end{lem}
\begin{proof}
	Integrating by parts, we have, for any $v\in H^1_{0,\omega}(\Omega)$,
	\begin{eqnarray}
		\int_{\Omega} b xvv'= \frac{1}{2}bx v^2\Big\vert_{\partial \Omega}-\frac{1}{2}\int_{\Omega}bv^2dx.
	\end{eqnarray}	
Since $v\in H^1_{0,\omega}(\Omega)$, then $v\big\vert_{\partial \Omega}=0$. Thereby, we have
\begin{eqnarray}
	\int_{\Omega} bxvv'= -\frac{1}{2}\int_{\omega} bv^2dx.
\end{eqnarray}
Indeed, we also have
\begin{eqnarray}
A(v,v;t)	 & = & \Big(ax^2v',v'\Big)+\Big(bxv,v'\Big)+\Big(cv,v\Big) \nonumber\\      
	&   &  \nonumber\\
	& = & \Big(ax^2v',v'\Big) - \frac{1}{2}\Big(bv,v\Big)+\Big(c v,v\Big) \nonumber\\
	&   &  \nonumber\\
%	& =  & \Big(ax^2v',v'\Big) +\frac{1}{2}\Big(\big(-b+2c\big)v,v\Big) \nonumber\\
%	&   & \nonumber \\
	& = & \Big(ax^2v',v'\Big)+\frac{1}{2}\Big(\big(3r+2\beta-\sigma^2\big)v,v\Big). \nonumber
	\end{eqnarray}
	%Besides, for
By using (\ref{ass-sig}), we have  $2\beta-\sigma^2>0$ which leads to 
	\begin{eqnarray}
	%&   & \nonumber \\
A(v,v;t) & \geq & \frac{1}{2} \times \sigma^2\Big(x^2v',v'\Big)+\frac{1}{2}\times 3r\Big(v,v\Big) \label{asp} \\
	&   &  \nonumber\\
	& \geq & \frac{1}{2}\min\big(\underline{\sigma}^2,3r\big)\Bigg[\Big(x^2v',v'\Big)+\Big(v,v\Big)\Bigg]  \nonumber \\
	&    &  \nonumber\\
A(v,v;t) &   \geq & C \big\vert\big\vert v \big\vert\big\vert_{1,\omega}^2. 
\end{eqnarray}
%Besides, for
%By using (\ref{ass-sig}), we have  $2\beta-\sigma^2>0$ which leads to (\ref{asp})
\end{proof}
\begin{theo}
 Under  Asumption \ref{assum1},\textbf{Problem 1} have a unique solution.
\end{theo}
\begin{proof}
The proof uses  lemma \ref{lem1} and can be found in \cite[Theorem 1.33]{existence}.
\end{proof}
\section{The finite volume formulation}
\label{finfit}
\subsection{Finite volume grid and discrete representation of the exact solution}
%\label{finfit}
Let  $\Omega$ be subdivided into sub-intervals as follows:
\begin{equation}
	\Omega_i=[x_{i};x_{i+1}]~~~~i=0,..,N, 
\end{equation} 
with~~$0=x_0<x_1<x_2<\ldots<x_N<x_{N+1}=x_{max}~~~$ and~~~$h_i=x_{i+1}-x_i$. We also  defined the following mid-points of the intervals $\Omega_i$ by
\begin{eqnarray*}
x_{i-\frac{1}{2}}=\frac{x_{i-1}+x_i}{2}~~~\text{and}~~x_{i+\frac{1}{2}}=\frac{x_i+x_{i+1}}{2}~~~~~~~\text{for}~~i=0,1,2,\dots,N+1,
\end{eqnarray*}
 with $x_{-\frac{1}{2}}=x_0$ and $x_{N+\frac{3}{2}}=x_{\max}$. These mid-points help us to define  another partition  $K_i$  of $\Omega$,  called dual partition, defined by 
\begin{eqnarray*}
	K_i=[x_{i-\frac{1}{2}};x_{i+\frac{1}{2}}]~~~~~~~~~~l_i=x_{i+\frac{1}{2}}-x_{i-\frac{1}{2}}~~~~~~~~~~i=0,..,N.
\end{eqnarray*}
\begin{Ass}{[Local quasi-uniformity of the spatial mesh]}\\
\label{assum2}
There exists a constant $c>0$ such that
	\begin{eqnarray}
	\label{ass-sp}
	%\frac{l_{i+1}}{c} \leq    l_i   \leq cl_{i+1},~~~~~and~~~~
	\frac{l_{i+1}}{c} \leq   l_i   \leq cl_{i+1}~~~~~i=0,1,\ldots,N.                   
	\end{eqnarray}
\end{Ass}
Since  the dual partition  $K_i$ is linked to the partition $\Omega_i$, Assumption \ref{assum2} implies that 
\begin{eqnarray}
\frac{h_{i+1}}{c} \leq    h_i   \leq ch_{i+1}\;\;\; i=0,1,\ldots,N.
\end{eqnarray}
We can now apply the finite volume method by integrating (\ref{self-adj}) over each interval $K_i$ for $i=1,\ldots,N$.
\begin{equation}
\label{eq-diff-int}
\int_{K_i}\frac{\partial u}{\partial \tau}dx-\int_{K_i}\frac{\partial}{\partial x}\Bigg[ax^2\frac{\partial u}{\partial x} +bxu\Bigg]dx+\int_{K_i}cudx=\int_{K_i} f(x,t) dx.
\end{equation}
Multiplying \ref{eq-diff-int} by an arbitrary real number $v_i$ for each $i=1,\ldots,N$ and summing them up, we get 
\begin{equation}
\label{eq-diff-int-1}
\sum_{i=1}^N\int_{K_i}\frac{\partial u}{\partial \tau}v_idx-\sum_{i=1}^N\int_{K_i}\frac{\partial}{\partial x}\Bigg[ax^2\frac{\partial u}{\partial x}+bxu\Bigg]v_idx+\sum_{i=1}^N\int_{K_i}cuv_idx=\sum_{i=1}^N\int_{K_i} f(x,t)v_i dx.
\end{equation}
Besides, for a function $v\in C(\overline{\Omega})$ we define the mass lumping operator $L_h$ define as follows:
\begin{eqnarray}
L_h:C(\overline{\Omega}) & \longrightarrow &  L^{\infty}(\Omega)   \nonumber \\
 %&      &  \\  
v      & \mapsto  &  L_hv\big\vert_{K_i}:=v(x_i),~~~~~~~~~i=1,\ldots,N. \nonumber
\end{eqnarray}
Moreover, if the function $v$ satisfies homogeneous Dirichlet boundary conditions, we have $L_hv\big\vert_{\partial \Omega}=0$. Then using the operator $L_h$, we can re-write (\ref{eq-diff-int-1}) as follows:
\begin{equation}
\label{dexact}
\big(-\dot{u}(t),L_hv\big) +\hat{a}_h\big(u(t),v;t\big)=\big(f(t),L_hv\big),
\end{equation} 
where
\begin{equation}
\hat{a}_h\big(\omega,v;t\big):=\sum_{i=1}^N\Big( F(\omega(x_{i+\frac{1}{2}}))-F\big(\omega(x_{i-\frac{1}{2}})\big)\Big)L_hv(x_i)+ \big(c(t)\omega,L_h\omega\big),
\end{equation}
with $F$ denoting the continuous flux defined as follows:
\begin{equation}
F\big(\omega(x_{i+\frac{1}{2}})\big):=-a(t)x_{i+\frac{1}{2}}^2\frac{\partial \omega}{\partial x}(x_{i+\frac{1}{2}})-b(t)x_{i+\frac{1}{2}}\omega\big(x_{i+\frac{1}{2}}\big)~~~~~~~i=0,...,N.
\end{equation}
Note that  \eqref{dexact} is a representation of the exact solution on the dual partition $K_i$ and will play a  key role in our error analysis.
%We  are now ready  to approximate $u$ in the  dual partition $K_i$.  Indeed we denote by
%%\begin{eqnarray}
%$u(x_i)\approx u_i$. 
%%\end{eqnarray}
%To approximate some integral terms of (\ref{eq-diff-int}), we use the mid-quadrature rule as follows: 
%\begin{equation}
%\int_{K_i}\frac{\partial u}{\partial \tau}dx \approx l_i\frac{\partial u_i}{\partial \tau}
%\end{equation}
%\begin{equation}
%\int_{K_i}cudx \approx l_icu_i
%\end{equation}
%and
%\begin{equation}
%\int_{K_i} f(x,t) dx \approx l_if_i~~~~~~~~~~f_i=f(x_i,t)
%\end{equation}
%Since the Black-Scholes equation (\ref{BS-PDE}) is  one dimensional in space, instead of the Multi-Point Flux Approximation (MPFA) method we use the Two-Point Flux Approximation (TPFA) method to approximate the following  second term of (\ref{eq-diff-int})   
%\begin{equation}
%\label{diff-term}
%\int_{K_i}\frac{\partial}{\partial x}\Bigg[ax^2\frac{\partial u}{\partial x}\Bigg]dx
%\end{equation}
\subsection{The Two Point Flux Approximation (TPFA) method}
\label{TPFA}
We  are now ready  to approximate $u$ in the  dual partition $K_i$.  Indeed we denote by
%\begin{eqnarray}
$u(x_i,t)\approx u_i$. 
%\end{eqnarray}
To approximate some integral terms of (\ref{eq-diff-int}), we use the mid-quadrature rule as follows: 
\begin{eqnarray}
\int_{K_i}\frac{\partial u}{\partial \tau}dx \approx l_i\frac{d u_i}{d \tau},\;\;\;
%\end{equation}
%\begin{equation}
\int_{K_i}cudx \approx l_icu_i,\;\;
%\end{equation}
%and
%\begin{equation}
\int_{K_i} f(x,t) dx \approx l_if_i~~~~~~~~~~f_i=f(x_i,t).
\end{eqnarray}
Since the Black-Scholes equation (\ref{BS-PDE}) is  one dimensional in space, instead of the Multi-Point Flux Approximation (MPFA) method we use the Two-Point Flux Approximation (TPFA) method to approximate the following  second term of (\ref{eq-diff-int})   
\begin{equation}
\label{diff-term}
\int_{K_i}\frac{\partial}{\partial x}\Bigg[ax^2\frac{\partial u}{\partial x}\Bigg]dx.
\end{equation}

The Two-Point Flux Approximation (TPFA) method is used to approximate the term (\ref{diff-term}) as follows:
\begin{eqnarray}
	\int_{K_i}\frac{\partial}{\partial x}\Bigg[a(t)x^2\frac{\partial u}{\partial x}\Bigg]dx & = & \Bigg[a(t)x^2\frac{\partial u}{\partial x}\Bigg]_{x_{i-\frac{1}{2}}}^{x_{i+\frac{1}{2}}} 
	%&    & \nonumber\\
	 =  a(t)x^2{\frac{\partial u}{\partial x}}\Bigg\vert_{x_{i+\frac{1}{2}}}-a(t)x^2{\frac{\partial u}{\partial x}}\Bigg\vert_{x_{i-\frac{1}{2}}}. \nonumber
	%&    & \nonumber\\
	%&    &  \nonumber
\end{eqnarray}
Let us set 
\begin{eqnarray}
\label{exp}
H(x)=k(x,t)\frac{\partial u}{\partial x}~~~~~~~~~~~~~with~~~~~k(x,t)=\frac{1}{2}\sigma^2(t)x^2. 
\end{eqnarray}
Over an interval $K_i$, $k(x,t)$ in (\ref{exp}) will be replaced by its average value defined as follows:
\begin{equation}
k_i= \frac{1}{\mathrm{meas}(K_i)}\int_{K_i}\frac{1}{2}\sigma^2(t)x^2 dx= \frac{1}{6}\sigma^2(t)\frac{x_{i+\frac{1}{2}}^3-x_{i-\frac{1}{2}}^3}{x_{i+\frac{1}{2}}-x_{i-\frac{1}{2}}},
\end{equation}
where $\mathrm{meas}(K_i)$ is the length of the interval $K_i$.
\begin{figure}[!h]
	\label{fig:itv}
	\centering
\begin{tikzpicture}
	\draw (0,0)--(12,0);
	\draw[black,fill=black] (0,0) circle (0.1);
	\draw[black,fill=black] (2,0) circle (0.1);
	\draw[black,fill=black] (4,0) circle (0.1);
	\draw[black,fill=black] (6,0) circle (0.1);
	\draw[black,fill=black] (2,-1) circle (0.1);
	\draw[black,fill=black] (8,0) circle (0.1);
	\draw[black,fill=black] (10,0) circle (0.1);
	\draw[black,fill=black] (12,0) circle (0.1);
	\node[below,black] at (0,-0.2){$x_{i-1}$};
	\node[below,black] at (4,-0.2){$x_{i}$};
	\node[below,black] at (8,-0.2){$x_{i+1}$};
	\node[above,black] at (2,0.2){$x_{i-\frac{1}{2}}$};
	\node[above,black] at (6,0.2){$x_{i+\frac{1}{2}}$};
	\node[above,black] at (10,0.2){$x_{i+\frac{3}{2}}$};
	\node[below,black] at (12,-0.2){$x_{i+2}$};
	\draw (2,-1)--(10,-1);
	\node[below,black] at (4,-1.2){$K_{i}$};
	\node[below,black] at (8,-1.2){$K_{i+1}$};
	\draw[black,fill=black] (2,-1) circle (0.1);
	\draw[black,fill=black] (6,-1) circle (0.1);
	\draw[black,fill=black] (10,-1) circle (0.1);
\end{tikzpicture}
\caption{Interval}
\end{figure}
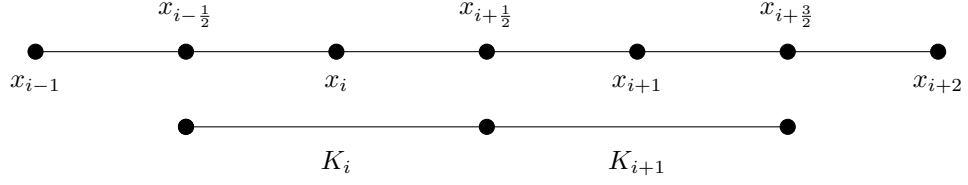
Thereby from  Figure (\ref{fig:itv}), $H_{i+\frac{1}{2}}$ can be evaluated at each side  of $x_{i+\frac{1}{2}}$ as follows:
\begin{eqnarray}
\label{Hi}
	H_{i+\frac{1}{2}}=2k_i\frac{u_{i+\frac{1}{2}}-u_i}{l_i},\;\;
	%\nonumber\\
%\label{Hi+1}
	H_{i+\frac{1}{2}}=2k_{i+1}\frac{u_{i+1}-u_{i+\frac{1}{2}}}{l_{i+1}}.
\end{eqnarray} 
Using the continuity of the flux at the interface $x_{i+\frac{1}{2}}$, we can equate  the two terms  in(\ref{Hi}). This leads to 
\begin{equation}
u_{i+\frac{1}{2}}=\frac{\frac{k_i}{l_i}u_i+\frac{k_{i+1}}{l_{i+1}}{u_{i+1}}}{\frac{k_i}{l_i}+\frac{k_{i+1}}{l_{i+1}}}.
\end{equation}
By setting   $T_i=\frac{k_i}{l_i}$, we can rewrite $H_{i+\frac{1}{2}}$  in (\ref{Hi}) as follows:
\begin{eqnarray}
\label{trans}
H_{i+\frac{1}{2}}=\tau_{i+\frac{1}{2}}(u_{i+1}-u_i),\;\;
%\end{equation}
%where
%\begin{equation}
%\label{trans}
\tau_{i+\frac{1}{2}}=\frac{2T_iT_{i+1}}{T_i+T_{i+1}}.
\end{eqnarray}
Besides, to approximate the third integral term of (\ref{eq-diff-int}), we use the upwind method as follows:
\begin{eqnarray}
\int_{K_i}\frac{\partial}{\partial x}\Big[b(t)xu\Big]dx & =  & \Big[b(t)xu\Big]_{x_{i-\frac{1}{2}}}^{x_{i+\frac{1}{2}}} 
      \approx b(t)\Big(x_{i+\frac{1}{2}}u_{i+\frac{1}{2}}-x_{i-\frac{1}{2}}u_{i-\frac{1}{2}}\Big) 
\end{eqnarray}
with 
\begin{equation}
u_{i+\frac{1}{2}}=\left\lbrace \begin{array}{lcr}
u_i  & ~~~~~~if~~b>0\\
%&   & \\
u_{i+1} &~~~~~~if~~b<0
\end{array} 
\right.	
\end{equation}
Therefore, the discrete formulation of \eqref{eq-diff-int} is:
\begin{equation}
\label{form-TPFA}
l_i\frac{d u_i}{d\tau}+ F_h(u_{i+\frac{1}{2}})-F_h(u_{i-\frac{1}{2}})-l_icu_i=l_if_i~~~~~~i=1,...,N,
\end{equation} 
where 
\begin{eqnarray}
\label{disc-flux-TPFA}
F_h(u_{i+\frac{1}{2}})   =  -\tau_{i+\frac{1}{2}}(u_{i+1}-u_i)-x_{i+\frac{1}{2}}\Big(b^+u_i+b^-u_{i+1}\Big),\;\;
%&    &  \\
F_h(u_{i-\frac{1}{2}})   = -\tau_{i-\frac{1}{2}}(u_i-u_{i-1})-x_{i-\frac{1}{2}}\Big(b^+u_{i-1}+b^-u_i\Big),
\end{eqnarray}
with  $b^+=\max(b,0)$ and $b^-=\min(b,0)$.
Moreover, in order to analyse the above scheme, it is convenient to rewrite it in a discrete variational form. Multiplying equation  \eqref{form-TPFA} by arbitrary real numbers $v_i$ and summing the result over all the intervals $K_i$ of $\Omega$, we get:
\begin{equation}
\label{var-form-TPFA}
\sum_{i=1}^N l_i\frac{d u_i}{d \tau}v_i+ \sum_{i=1}^N \Big(F_h(u_{i+\frac{1}{2}})-F_h(u_{i-\frac{1}{2}})\Big)v_i +c\sum_{i=1}^Nl_iu_iv_i=\sum_{i=1}^Nl_if_iv_i.
\end{equation}
Let us denote by $V_h \subset H^1_{\omega}(\Omega)$ the space of continuous functions that are piecewise continuous over the  the grid $(K_i)$ of $\Omega$. Thereby, the TPFA method \eqref{form-TPFA} 
% is used for the spatial discretization ,  
is  equivalent to
%we can define over $V_h$ the following bilinear form:
\begin{eqnarray}
\label{bila}
a_h(u_h,v_h)=a_h^1(u_h,v_h)+a_h^2(u_h,v_h)+c\sum_{i=1}^Nl_iu_iv_i~~~~~u_h,v_h\in V_h,
\end{eqnarray}
with
\begin{equation}
\label{bila-1}
a_h^1(u_h,v_h) = \sum_{i=1}^N \Bigg[-\tau_{i+\frac{1}{2}}(u_{i+1}-u_i)+\tau_{i-\frac{1}{2}}(u_i-u_{i-1})\Bigg]v_i,
\end{equation}
and
\begin{equation}
\label{bila-2}
a_h^2(u_h,v_h)=\sum_{i=1}^N \Bigg[-x_{i+\frac{1}{2}}\Big(b^+u_{i+1}+b^-u_{i}\Big) +x_{i-\frac{1}{2}}\Big(b^+u_{i-1}+b^-u_{i}\Big)\Bigg]v_i.
\end{equation}
Let us notice that we can rewrite the bilinear form in (\ref{bila}) as:
\begin{equation}
\label{bilaflux-TPFA}
a_h(u_h,v_h)=\sum_{i=1}^N \Big(F_h(u_{i+\frac{1}{2}})-F_h(u_{i-\frac{1}{2}})\Big)v_i+c\sum_{i=1}^Nl_iu_iv_i,
\end{equation}
where $F_h$ is the discrete flux in given in  \eqref{disc-flux-TPFA}.
\subsection{The fitted Two Point Flux Approximation method}
\label{fit-TPPFA}
Since the PDE  (\ref{self-adj}) is degenerated when $x$ approaches zero, the second term of (\ref{eq-diff-int}), at the point $x=x_{1/2}$, will be approximated  using the fitted finite volume method introduced \cite{wang2004novel}. 
The fitted finite volume method will consist to solve a two-point value problem over the interval $K_1$. Thereby,  from \cite{wang2004novel}, we have 
\begin{eqnarray}
ax^2\frac{ \partial u}{ \partial x} +bxu\Bigg\vert_{x_{1/2}} &  \approx  & \frac{1}{4}x_1[(a+b)u_1-(a-b)u_0].
\end{eqnarray}
On the the rest on the study domain ($K_i,~i=2,..,N$), we apply the TPFA method coupled to the upwind method introduced in section \ref{TPFA}.
Therefore, the discrete approximation of (\ref{eq-diff-int})  by the fitted TPFA  method is given
\begin{equation}
\label{form-fitTPFA}
l_i\frac{d u_i}{d\tau}+ G_h(u_{i+\frac{1}{2}})-G_h(u_{i-\frac{1}{2}})-l_icu_i=l_if_i~~~~~~i=1,...,N,
\end{equation} 
where 
\begin{eqnarray}
\label{disc-flux-fitTPFA}
G_h(u_{i+\frac{1}{2}})  & = & -\tau_{i+\frac{1}{2}}(u_{i+1}-u_i)-x_{i+\frac{1}{2}}\Big(b^+u_i+b^-u_{i+1}\Big) \nonumber\\
%&    & \nonumber \\
G_h(u_{i-\frac{1}{2}})  & = & -\tau_{i-\frac{1}{2}}(u_i-u_{i-1})-x_{i-\frac{1}{2}}\Big(b^+u_{i-1}+b^-u_i\Big),~~~~i\neq 1  \\
%&    &  \nonumber\\
G_h(u_{1/2}) & =  &  -\frac{1}{4}x_1(a+b)u_1, \nonumber
\end{eqnarray}
with $b^+=\mathrm{max}(b,0)$ and $b^-=\mathrm{min}(b,0)$.
Thereby, the discrete spatial formulation is given by
\begin{equation}
\label{var-form-fitTPFA}
\sum_{i=1}^N l_i\frac{ d u_i}{ d\tau}v_i+ \sum_{i=1}^N \Big(G_h(u_{i+\frac{1}{2}})-G_h(u_{i-\frac{1}{2}})\Big)v_i +c\sum_{i=1}^Nl_iu_iv_i=\sum_{i=1}^Nl_if_iv_i.
\end{equation} 
We define the corresponding bilinear form $b_h$ by:
\begin{equation}
\label{bilb}
b_h(u_h,v_h)=b_h^1(u_h,v_h)+b_h^2(u_h,v_h)+c\sum_{i=1}^Nl_iu_iv_i~~~~~u_h,v_h\in V_h,
\end{equation}
with
\begin{equation}
\label{bilb-1}
b_h^1(u_h,v_h) = \sum_{i=1}^N \Big[-\tau_{i+\frac{1}{2}}(u_{i+1}-u_i)\Big]+\sum_{i=2}^N\Big[\tau_{i-\frac{1}{2}}(u_i-u_{i-1})\Big]v_i,
\end{equation}
and
\begin{equation}
\label{bilb-2}
b_h^2(u_h,v_h)  =   \frac{1}{4}x_1(a+b)u_1^2+\sum_{i=1}^N \Big[-x_{i+\frac{1}{2}}\Big(b^+u_i+b^-u_{i+1}\Big)\Big]v_i +\sum_{i=2}^N \Big[ x_{i-\frac{1}{2}}\Big(b^+u_{i-1}+b^-u_{i}\Big)\Bigg]v_i.
\end{equation}
The bilinear form $b_h$ in (\ref{bilb}) can be rewritten  as:
 \begin{equation}
\label{bilbflux-fitTPFA}
b_h(u_h,v_h)=\sum_{i=1}^N \Big(G_h(u_{i+\frac{1}{2}})-G_h(u_{i-\frac{1}{2}})\Big)v_i+c\sum_{i=1}^Nl_iu_iv_i,
\end{equation}
where   $G_h$ is the discrete flux given by the fitted TPFA in \eqref{disc-flux-fitTPFA}.
\subsection{Coercivity and Flux consistency for  TPFA and fitted TPFA}
 Let us  denote by  $  \big( \cdot , \cdot \big)_{h} $ the scalar product on $C(\overline{\Omega})\supset V_h$ by 
\begin{equation}
\big( u,v\big)_{h}=(L_hu,L_hv) = \sum_{i=1}^N l_iu_iv_i~~~~~~~~~~u,v\in C(\overline{\Omega}),
\end{equation}
 and its corresponding norm $\vert\vert \cdot \vert\vert_{0,h}$ by
  \begin{equation}
 \label{norm-0h}
 ||v||_{0,h}^2 = \sum_{i=1}^N l_i v_i^2,
 \end{equation}
 we define the discrete $H_0^1(\Omega)$ norm by 
\begin{equation}
\label{norm-1h}
\vert\vert u_h \vert\vert^2_{0,\omega} =\sum_{i=1}^N \tau_{i+\frac{1}{2}}\vert u_{i+1}-u_i \vert^2,
\end{equation}
and  weighted discrete $H^1_{\omega}-$ norm is then
\begin{equation}
\label{w-norm}
\vert\vert u_h \vert\vert^2_{\omega,d} = \vert\vert u_h \vert\vert^2_{0,\omega}+ \vert\vert u_h \vert\vert^2_{0,h}.
\end{equation}
Indeed it is easy to show that  $\vert \vert .  \vert \vert_{0,\omega}$ is a semi-norm in $V_h$  since $\tau_{i+\frac{1}{2}}>0$.
\begin{theo}{[Coercivity of  bilinear forms]}\\
	Under the regularity of the  mesh (see Assumption \ref{assum2}) and Assumption \ref{assum1}, there exists a constant $\alpha>0$ independent of h such that,
	% when the TPFA method (see section \ref{TPFA}) we have:
	\begin{equation}
	\label{coerc-TPFA}
	  a_h(u_h,u_h) \geq \alpha \vert\vert u_h \vert\vert_{\omega,d} ~~~~~~~~\forall u_h \in V_h.
	\end{equation}
	where $a_h$ is the bilinear form given by \eqref{bila} for the TPFA method.
Simlilarly, when the fitted TPFA method \eqref{form-fitTPFA}
%(see section \ref{fit-TPPFA}) 
is  used for the space discretization, there  exists a constant $\gamma>0$ independent of h such that
\begin{equation}
\label{coerc-fitTPFA}
b_h(u_h,u_h) \geq \gamma \vert\vert u_h \vert\vert_{\omega,d} ~~~~~~~~\forall u_h \in V_h.
\end{equation}
where the bilinear form $b_h$ is given by  \eqref{bilb}. 
\end{theo}
\begin{remark}
Note that using  the coercivity properties in \eqref{coerc-TPFA} and \eqref{coerc-fitTPFA}, with the fact that the linear mapping\\ $ v \rightarrow (f, v)_h$ is continuous in $V_h$, the existence and uniqueness of the discrete solution $u_h$ is ensure for both the TPFA and fitted TPFA methods in \eqref{var-form-TPFA} and \eqref{var-form-fitTPFA}.
The proof is done exactly as for  the continuous case (see \cite[Theorem 1.33]{existence}).
\end{remark}
\begin{proof}
Here we distinguish two cases which are:
%\begin{itemize}
%	\item 
\underline{$1^{st} case$:} The standard TPFA method  is used for the spatial discretization
In this case, the discrete flux is given by \eqref{disc-flux-fitTPFA} and the corresponding bilinear form  \eqref{bila} is 
\begin{equation*}
a_h(u_h,v_h)=a_h^1(u_h,v_h)+a_h^2(u_h,v_h)+c\sum_{i=1}^Nl_iu_iv_i~~~~~u_h,v_h\in V_h
\end{equation*}
with
\begin{equation*}
a_h^1(u_h,v_h) = \sum_{i=1}^N \Bigg[-\tau_{i+\frac{1}{2}}(u_{i+1}-u_i)+\tau_{i-\frac{1}{2}}(u_i-u_{i-1})\Bigg]v_i
\end{equation*}
and
\begin{equation}
a_h^2(u_h,v_h)=\sum_{i=1}^N \Bigg[-x_{i+\frac{1}{2}}\Big(b^+u_{i+1}+b^-u_{i}\Big) +x_{i-\frac{1}{2}}\Big(b^+u_{i-1}+b^-u_{i}\Big)\Bigg]v_i
\end{equation}
thereby, we have:
\begin{eqnarray}
\label{bil-f-11}
a_h^1(u_h,u_h) & = & \sum_{i=1}^N \Bigg[-\tau_{i+\frac{1}{2}}\big(u_{i+1}-u_i\big)+\tau_{i-\frac{1}{2}}\big(u_i-u_{i-1}\big)\Bigg]u_i \nonumber\\
%&    & \nonumber\\
%& =  & \sum_{i=1}^N-\tau_{i+\frac{1}{2}}\big( u_{i+1}-u_i\big)u_i+\sum_{i=1}^N\tau_{i-\frac{1}{2}}\big(u_i-u_{i-1}\big)u_i\nonumber\\
%&   & \nonumber\\
& =  & \sum_{i=1}^N-\tau_{i+\frac{1}{2}}\big( u_{i+1}-u_i\big)u_i+\sum_{k=0}^{N-1}\tau_{k+\frac{1}{2}}\big(u_{k+1}-u_{k}\big)u_{k+1}\nonumber\\
%&    & \nonumber\\
%&  = & \tau_{1/2}\big(u_1-u_0\big)u_1+\sum_{i=1}^{N-1}-\tau_{i+\frac{1}{2}}\big( u_{i+1}-u_i\big)\big(u_i-u_{i+1}\big)-\tau_{N+\frac{1}{2}}\big(u_{N+1}-u_N\big)u_N\nonumber\\
%&    & \nonumber\\
%& = &  \tau_{1/2}\big(u_1-u_0\big)\big(u_1-u_0\big)+\sum_{i=1}^{N-1}\tau_{i+\frac{1}{2}}\big( u_{i+1}-u_i\big)^2+\tau_{N+\frac{1}{2}}\big(u_{N+1}-u_N\big)\big(u_{N+1}-u_N\big)\nonumber\\
%&    & \nonumber\\
& = &\tau_{1/2}\big(u_1-u_0\big)u_1+\sum_{i=1}^{N-1}\tau_{i+\frac{1}{2}}\big( u_{i+1}-u_i\big)\big( u_{i+1}-u_i\big)-\tau_{N+\frac{1}{2}}\big(u_{N+1}-u_N\big)u_N\nonumber\\
%&    & \nonumber\\
%&    & \nonumber \\
& =   & \tau_{1/2}\big(u_1-u_0\big)\big(u_1-u_0\big) + \sum_{i=1}^{N-1}\tau_{i+\frac{1}{2}}\big( u_{i+1}-u_i\big)^2+\tau_{N+\frac{1}{2}}\big(u_{N+1}-u_N\big)\big(u_{N+1}-u_N\big)\nonumber\\
%&    & \nonumber \\
&   =  & \tau_{1/2}\big(u_1-u_0\big)^2+\sum_{i=1}^{N}\tau_{i+\frac{1}{2}}\big( u_{i+1}-u_i\big)^2 \nonumber \\
a_h^1(u_h,u_h)&   \geq &  \vert\vert u_h\vert\vert^2_{0,\omega},
\end{eqnarray}	
and also
\begin{eqnarray}
a_h^2(u_h,u_h) &  = & \sum_{i=1}^N \Bigg[-x_{i+\frac{1}{2}}\Big(b^+u_{i+1}+b^-u_{i}\Big) +x_{i-\frac{1}{2}}\Big(b^+u_{i-1}+b^-u_{i}\Big)\Bigg]u_i   \nonumber \\
&  = &  \sum_{i=1}^N -x_{i+\frac{1}{2}}\Big(b^+u_{i+1}+b^-u_{i}\Big)u_i + \sum_{i=1}^Nx_{i-\frac{1}{2}}\Big(b^+u_{i-1}+b^-u_{i}\Big)u_i \nonumber \\
& =  &  \sum_{i=1}^N -x_{i+\frac{1}{2}}\Big(b^+u_{i+1}+b^-u_{i}\Big)u_i + \sum_{i=0}^{N-1}x_{i+\frac{1}{2}}\Big(b^+u_i+b^-u_{i+1}\Big)u_{i+1} \nonumber\\
&  =  & x_{1/2}\Big(b^+u_0+b^-u_{1}\Big)u_{1}+ \sum_{i=1}^{N-1}x_{i+\frac{1}{2}}\Bigg[b^-\big(u_{i+1}^2-u_i^2\Big)\Bigg]-x_{N+\frac{1}{2}}\Big(b^+u_{N+1}+b^-u_{N}\Big)u_i \nonumber \\
&  = & -b^-x_{1/2}\big(u_0^2-u_1^2\big)-b^-\sum_{i=1}^{N-1}x_{i+\frac{1}{2}}
\Big(u_{i}^2-u_{i+1}^2\Big)-b^-x_{N+\frac{1}{2}}\Big(u_N^2-u_{N+1}^2\Big) \nonumber\\
&  = & -b^-\sum_{i=0}^{N}x_{i+\frac{1}{2}}
\Big(u_{i}^2-u_{i+1}^2\Big).\nonumber
\end{eqnarray}
We also have 
\begin{eqnarray}
\label{bil-f-12}
a_h^2(u_h,u_h) &  = & -b^- \Bigg[\sum_{i=0}^N \Big(x_{i-\frac{1}{2}}+l_i\Big)u_i^2-\sum_{i=0}^{N}
x_{i+\frac{1}{2}}u_{i+1}^2\Bigg] \nonumber\\
&   =  & -b^- \Bigg[\sum_{i=-1}^{N-1} \Big(x_{i+\frac{1}{2}}+l_{i+1}\Big)u_{i+1}^2-\sum_{i=0}^{N}
x_{i+\frac{1}{2}}u_{i+1}^2\Bigg] \nonumber\\
&  =  &  -b^-\Bigg[x_{-\frac{1}{2}}u_0^2+\sum_{i=-1}^{N-1}l_{i+1}u_{i+1}^2-x_{N+\frac{1}{2}}u_{N+1}^2\Bigg] \nonumber\\
&  = &  -b^-\sum_{i=0}^Nl_iu_i^2 \nonumber\\
a_h^2(u_h,u_h) & \geq  &  0,
\end{eqnarray}
then, using  \eqref{bil-f-11} and \eqref{bil-f-12}, we get:
\begin{eqnarray}
\label{bila1}
a_h(u_h,u_h) \geq \big\vert\big\vert u_h \big\vert\big\vert_{0,\omega}+c\sum_{i=1}^Nl_iu_i^2 .
\end{eqnarray}
Moreover, we have
\begin{equation}
c=2r+\beta -\sigma^2 > 0,
\end{equation}
therefore, this  yields to
\begin{equation}
a_h(u_h,u_h) \geq \alpha \Big( \big\vert\big\vert  u_h\big\vert\big\vert_{0,\omega}+\big\vert\big\vert u_h \big\vert\big\vert_{0,h}\Big),
\end{equation} 
with $\alpha=\min(1,c)$. Hence we get
\begin{equation}
a_h(u_h,u_h) \geq \alpha  \big\vert\big\vert  u_h\big\vert\big\vert_{\omega,d}.
\end{equation}
%\item 
\underline{$2^{nd}$ case} The fitted TPFA method is used for the spatial discretization.
  Here, $G_h$ defined in  \eqref{form-fitTPFA}, gives the discrete flux. Thereby, the corresponding bilinear form \eqref{bilb} is  given by
  \begin{equation}
b_h(u_h,v_h)=b_h^1(u_h,v_h)+b_h^2(u_h,v_h)+c\sum_{i=1}^Nl_iu_iv_i~~~~~u_h,v_h\in V_h
\end{equation}
  \begin{equation}
\label{}
b_h^2(u_h,v_h)  =   \frac{1}{4}x_1(a+b)u_1^2+\sum_{i=1}^N \Big[-x_{i+\frac{1}{2}}\Big(b^+u_i+b^-u_{i+1}\Big)\Big]v_i +\sum_{i=2}^N \Big[ x_{i-\frac{1}{2}}\Big(b^+u_{i-1}+b^-u_{i}\Big)\Bigg]v_i
\end{equation}
with
\begin{eqnarray}
\label{bilb-1-bound}
b_h^1(u_h,v_h) &  =  &  \sum_{i=1}^N \Big[-\tau_{i+\frac{1}{2}}\big(u_{i+1}-u_i\big)\Big]u_i+\sum_{i=2}^N\tau_{i-\frac{1}{2}}\Big[\big(u_i-u_{i-1}\big)\Big]u_i \nonumber\\
 & = & -\tau_{\frac{3}{2}}\big(u_{2}-u_1\big)u_1+ \sum_{i=2}^N \Bigg[-\tau_{i+\frac{1}{2}}\big(u_{i+1}-u_i\big)+\tau_{i-\frac{1}{2}}\big(u_i-u_{i-1}\big)\Bigg]u_i \nonumber\\
%&    & \nonumber\\
%& =  & \sum_{i=1}^N-\tau_{i+\frac{1}{2}}\big( u_{i+1}-u_i\big)u_i+\sum_{i=1}^N\tau_{i-\frac{1}{2}}\big(u_i-u_{i-1}\big)u_i\nonumber\\
%&   & \nonumber\\
& =  &-\tau_{3/2}\big(u_{2}-u_1\big)u_1+ \sum_{i=2}^N-\tau_{i+\frac{1}{2}}\big( u_{i+1}-u_i\big)u_i+\sum_{k=1}^{N-1}\tau_{k+\frac{1}{2}}\big(u_{k+1}-u_{k}\big)u_{k+1}\nonumber\\
%&    & \nonumber\\
%&  = & \tau_{1/2}\big(u_1-u_0\big)u_1+\sum_{i=1}^{N-1}-\tau_{i+\frac{1}{2}}\big( u_{i+1}-u_i\big)\big(u_i-u_{i+1}\big)-\tau_{N+\frac{1}{2}}\big(u_{N+1}-u_N\big)u_N\nonumber\\
%&    & \nonumber\\
%& = &  \tau_{1/2}\big(u_1-u_0\big)\big(u_1-u_0\big)+\sum_{i=1}^{N-1}\tau_{i+\frac{1}{2}}\big( u_{i+1}-u_i\big)^2+\tau_{N+\frac{1}{2}}\big(u_{N+1}-u_N\big)\big(u_{N+1}-u_N\big)\nonumber\\
%&    & \nonumber\\
& = &-\tau_{3/2}\big(u_{2}-u_1\big)u_1+\tau_{3/2}\big(u_2-u_1\big)u_2+\sum_{i=2}^{N-1}\tau_{i+\frac{1}{2}}\big( u_{i+1}-u_i\big)^2+\tau_{N+\frac{1}{2}}\big(u_{N+1}-u_N\big)^2\nonumber\\
%&    & \nonumber\\
%&    & \nonumber \\
& =   & \tau_{3/2}\big(u_2-u_1\big)^2 + \sum_{i=2}^{N}\tau_{i+\frac{1}{2}}\big( u_{i+1}-u_i\big)^2 \nonumber\\
%&    & \nonumber \\
              &   =  & \sum_{i=1}^{N}\tau_{i+\frac{1}{2}}\big( u_{i+1}-u_i\big)^2 \nonumber \\
b_h^1(u_h,u_h)&   = &  \vert\vert u_h\vert\vert^2_{0,\omega},
\end{eqnarray}	
and we have also
\begin{eqnarray}
b_h^2(u_h,u_h) & =  & \frac{1}{4}x_1(a+b)u_1^2+\sum_{i=1}^N \Big[-x_{i+\frac{1}{2}}\Big(b^+u_i+b^-u_{i+1}\Big)\Big]u_i +\sum_{i=2}^N \Big[ x_{i-\frac{1}{2}}\Big(b^+u_{i-1}+b^-u_{i}\Big)\Bigg]u_i \nonumber \\
&  =  & \frac{1}{4}x_1(a+b)u_1^2+\sum_{i=1}^N \Big[-x_{i+\frac{1}{2}}\Big(b^+u_i+b^-u_{i+1}\Big)\Big]u_i  +\sum_{i=1}^{N-1} \Big[ x_{i+\frac{1}{2}}\Big(b^+u_{i}+b^-u_{i+1}\Big)\Bigg]u_{i+1} \nonumber \\
& =  & \frac{1}{4}x_1(a+b)u_1^2+\sum_{i=1}^{N-1} x_{i+\frac{1}{2}}b^-\big(u_{i+1}^2-u_i^2\big) -x_{N+\frac{1}{2}}b^-u_N^2 \nonumber\\
b_h^2(u_h,u_h) &   = & \frac{1}{4}x_1(a+b)u_1^2+\sum_{i=1}^{N} x_{i+\frac{1}{2}}b^-\big(u_{i+1}^2-u_i^2\big)
\end{eqnarray}
\begin{enumerate}
	\item  \underline{$1^{st} case$}: $b\geq0$
		if $b>0$ then  $b^-=0$ so
	\begin{eqnarray}
	  b_h^2(u_h,u_h) & =  & \frac{1}{4}x_1(a+b)u_1^2 >0
	\end{eqnarray}
	\item  \underline{$2^{nd} case$}: $b<0$
		\begin{eqnarray}
	 b_h^2(u_h,u_h) &   = & \frac{1}{4}x_1(a+b)u_1^2+\sum_{i=1}^{N} x_{i+\frac{1}{2}}b^-\big(u_{i+1}^2-u_i^2\big) \nonumber \\
	              &   = & \frac{1}{4}x_1(a+b)u_1^2 -b^- \Bigg(\sum_{i=1}^N\Big(x_{i-\frac{1}{2}}+l_i\Big)u_i^2 -\sum_{i=1}^N x_{i+\frac{1}{2}}u_{i+\frac{1}{2}}^2\Bigg) \nonumber \\
	              &  =  & \frac{1}{4}x_1(a+b)u_1^2 -b^- \Bigg(\sum_{i=0}^{N-1}\Big(x_{i+\frac{1}{2}}+l_{i+1}\Big)u_{i+1}^2 -\sum_{i=1}^N x_{i+\frac{1}{2}}u_{i+\frac{1}{2}}^2\Bigg) \nonumber\\
	              &  = &   \frac{1}{4}x_1(a+b)u_1^2 -b^-\big(x_{1/2}+l_1\big)u_1^2 -\sum_{i=1}^{N} l_iu_i^2 -x_{N+\frac{1}{2}}u_{N+1}^2 \nonumber\\
	              &  =  & \frac{1}{4}x_1(a+b)u_1^2-x_{3/2}b^-u_1^2 -b^-\sum_{i=1}^N l_iu_i^2 \nonumber\\
	              &  =  & \frac{1}{4}x_1au_1^2+\big( \frac{1}{4}x_1 b- x_{3/2}b^-\big)u_1^2-b^-\sum_{i=1}^N l_iu_i^2 \nonumber\\
	              %&   = & \frac{1}{4}x_1\big(a-b^-\big)u_1^2-b^-\sum_{i=1}^N l_iu_i^2 \nonumber \\
b_h^2(u_h,u_h) &  \geq & 0.
	\end{eqnarray}
	\end{enumerate}
Thus for any $b$, we have
\begin{equation}
\label{bilb-2-bound}
b_h^2(u_h,u_h) \geq 0,
\end{equation}
thereby, using   \eqref{bilb-1-bound} and \eqref{bilb-2-bound} in \eqref{bilb},  we have:
\begin{eqnarray}
b_h(u_h,u_h) \geq \big\vert\big\vert u_h \big\vert\big\vert_{0,\omega}+c\sum_{i=1}^Nl_iu_i^2 .
\end{eqnarray}
Besides, we have
\begin{equation}
c=2r+\beta -\sigma^2 > 0,
\end{equation}
therefore, this yields to
\begin{equation}
b_h(u_h,u_h) \geq \gamma \Big( \big\vert\big\vert  u_h\big\vert\big\vert_{0,\omega}+\big\vert\big\vert u_h \big\vert\big\vert_{0,h}\Big),
\end{equation} 
with $\gamma=\min(1,c)$. Hence we get
\begin{equation}
\label{bilb-bound}
b_h(u_h,u_h) \geq \gamma  \big\vert\big\vert  u_h\big\vert\big\vert_{\omega,d}.
\end{equation}
%\end{itemize}
\end{proof}
%\begin{theo}{Error bounds}{\label{err-b}}\\
%	Let $u$ and $u_h$ be solutions of the continuous and discrete problem respectively. Then the following bounds holds
%	
%	\begin{equation}
%	\vert\vert u-u_h \vert\vert_{1,\omega} \leq C \cdot h
%	\end{equation}
%\end{theo}
%Before giving the proof of this theorem, let us state the following lemma which will help in demonstrating the theorem.
\begin{prop}{Flux consistency} \\
\label{flux-c}	
Let  $I_h$ be  the interpolation operator defines as follows:
\begin{eqnarray}
\label{interpol-op}
	I_h:C(\overline{\Omega}) & \longrightarrow &  V_h  \nonumber\\
	%&      & \\
	v      & \mapsto  &  I_hv(x):=\sum_{i=1}^N v(x_i)\phi_{x_i}(x),~~~~~~~~x\in \Omega  \nonumber
\end{eqnarray}
where $\lbrace \phi_{x_i}\rbrace_{i=1}^N$, with $\phi_{x_i}(x_j)=\delta_{ij}$, is the nodal basis to  $\lbrace x_i\rbrace_{i=1}^N,\,~~x_i\in K_i.$
Let  $F$ be the total (continuous) flux function defined for $\omega\in \mathcal{C}(\bar{\Omega})$ as 
\begin{equation}
%\omega\in \mathcal{C}(\bar{\Omega})
F\big(\omega(x_{i+\frac{1}{2}})\big):=-k(x_{i+\frac{1}{2}})\frac{\partial \omega}{\partial x}\big(x_{i+\frac{1}{2}}\big)-b x_{i+\frac{1}{2}} \omega\big(x_{i+\frac{1}{2}}\big)~~~~x_{i+\frac{1}{2}}\in\Omega_i.
\end{equation}	
When the TPFA method is applied for the spatial discretization i.e the discrete flux is given by $F_h$ 
defined in \eqref{disc-flux-TPFA},then for $\omega\in H_{0,\omega}^2(\Omega)$,  there is a positive constant $C_1$ such that
\begin{equation}
\label{fluxconst-TPFA}
\Big\vert F\big(w(x_{i+\frac{1}{2}})\big)-F_h\big(I_hw(x_{i+\frac{1}{2}})\big) \Big\vert                                     \leq  C_1 \int_{x_i}^{x_{i+1}}  \Big(\big\vert F'(\omega) \big\vert +\big\vert \omega' \big\vert + \big\vert \omega \big\vert \Big)dx ~~~~~~~~~~~~~~~~~~~i=0,...,N.
\end{equation}  
Similarly, when the fitted TPFA is applied for the spatial discretization i.e the discrete flux is given by $G_h$ defined in \eqref{disc-flux-fitTPFA},then for $\omega\in H_{0,\omega}^2(\Omega)$,  there is a positive constant $C_2$ such that
\begin{equation}
\label{fluxconst-fitTPFA}
\Big\vert F\big(w(x_{i+\frac{1}{2}})\big)-G_h\big(I_hw(x_{i+\frac{1}{2}})\big) \Big\vert                                     \leq  C_2 \int_{x_i}^{x_{i+1}}  \Big(\big\vert F'(\omega) \big\vert +\big\vert \omega' \big\vert + \big\vert \omega \big\vert \Big)dx.~~~~~~~~~~~~~~~~~~~i=0,...,N.
\end{equation}
\end{prop} 
Before proving the proposition \eqref{flux-c}, let us state the following lemma:
\begin{lem}
For $i=1,\ldots,N$, there exist two constants $C_3$ and $C_4$ independent of $h$  such that the transmissibility coefficient $\tau_{i+\frac{1}{2}}$ defined in (\ref{trans}) and its inverse are bounded as follows:
\begin{equation}
\label{tau-bounds}
\Big\vert \tau_{i+\frac{1}{2}}\Big\vert \leq  C_3,\;\;\quad \quad
%\end{equation}
%\begin{equation}
%\label{tau-inv-bound}
\frac{1}{\tau_{i+\frac{1}{2}}} \leq C_4 h_i~~~~~~~~~~~~~~~~~i=0,\ldots,N.
\end{equation}
\end{lem} 
  \begin{proof}
 In one hand, we have
\begin{eqnarray*}
	\Big\vert \tau_{i+\frac{1}{2}}\Big\vert  & =  & \frac{k_ik_{i+1}}{l_{i+1}k_i+l_ik_{i+1}}  \\
	%&    & \\
	&   = & \frac{\sigma^2}{6} \frac{\Big(x_{i+\frac{1}{2}}^3-x_{i-\frac{1}{2}}^3\Big)\Big(x_{i+\frac{3}{2}}^3-x_{i+\frac{1}{2}}^3\Big)}{l_{i+1}^2\Big(x_{i+\frac{1}{2}}^3-x_{i-\frac{1}{2}}^3\Big)+l_i^2\Big(x_{i+\frac{3}{2}}^3-x_{i+\frac{1}{2}}^3\Big)}\\
	%&    & \\
	%	&  \leq  & \frac{\sigma^2}{6}\frac{\Big(x_{i+\frac{1}{2}}^3-x_{i-\frac{1}{2}}^3\Big)\Big(x_{i+\frac{3}{2}}^3-x_{i+\frac{1}{2}}^3\Big)}{l_i^2\Big(x_{i+\frac{3}{2}}^3-x_{i+\frac{1}{2}}^3\Big)}\\
	%	&     & \\
	&  \leq  & \frac{\sigma^2}{6}\frac{\Big(x_{i+\frac{1}{2}}^3-x_{i-\frac{1}{2}}^3\Big)}{l_i^2},
\end{eqnarray*}
so
\begin{eqnarray*}
	\Big\vert \tau_{i+\frac{1}{2}}\Big\vert    &  \leq  &  \frac{\sigma^2}{6}\frac{\Big(x_{i+\frac{1}{2}}^3-x_{i-\frac{1}{2}}^3\Big)}{\Big(x_{i+\frac{1}{2}}-x_{i-\frac{1}{2}}\Big)^2} \\
	%&      & \\
	&  \leq   & \frac{\sigma^2}{6} \frac{x_{i+\frac{1}{2}}^2+x_{i+\frac{1}{2}}x_{i-\frac{1}{2}}+x_{i-\frac{1}{2}}^2}{x_{i+\frac{1}{2}}-x_{i-\frac{1}{2}}} \\
	%	&    & \\
	%	&    & \\%and
	%
	%
	%\begin{equation}
	%\vert\vert u_h \vert\vert^2_{2,h}=\frac{1}{2}\sum_{i=1}^N \vert a_i \vert \cdot \vert u_{i+1}-u_i \vert^2
	%\end{equation}
	%	&  \leq & \frac{\sigma^2}{6} \frac{x_{i+\frac{1}{2}}^2\Bigg(1+\frac{x_{i-\frac{1}{2}}}{x_{i+\frac{1}{2}}}+\left(\frac{x_{i-\frac{1}{2}}}{x_{i+\frac{1}{2}}}\right)^2\Bigg)}{x_{i+\frac{1}{2}}\Bigg(1-\frac{x_{i-\frac{1}{2}}}{x_{i+\frac{1}{2}}}\Bigg)}\\
	%	&    & \\
	&  \leq  & \frac{\sigma^2}{6} \frac{x_{i+\frac{1}{2}}\Bigg(1+\frac{x_{i-\frac{1}{2}}}{x_{i+\frac{1}{2}}}+\left(\frac{x_{i-\frac{1}{2}}}{x_{i+\frac{1}{2}}}\right)^2\Bigg)}{1-\frac{x_{i-\frac{1}{2}}}{x_{i+\frac{1}{2}}}}\\
	%&     & \\
	&  \leq & \frac{\sigma^2 x_{\max}}{6} \frac{1}{1-X_i} \Bigg(1+X_i+X_i^2\Bigg).
\end{eqnarray*}
Let us define 
\begin{equation}
X_{i}=\dfrac{x_{i-\frac{1}{2}}}{x_{i+\frac{1}{2}}}~~~~~~~with~~~0 < X_i< 1.
\end{equation}
Thereby, using the Taylor expansion, we have
\begin{equation}
\label{xtayl-O}
\dfrac{1}{1-X_i}= 1+ X_i+\mathcal{O}(X_i^2).
\end{equation}
Then there exists a constant $M_1$ such that 
\begin{equation}
\label{xtayl-const}
\dfrac{1}{1-X_i} \leq 1+ X_i+\mathbf{M_1}X_i^2 \leq 2+\mathbf{M_1}.
\end{equation}
Since $0\leq X_i\leq 1$, we have also
\begin{equation}
\label{xtayl-const-M}
0 \leq 1+X_i+X_i^2 \leq 3,
\end{equation}
thus using  $X_i$ in (\ref{xtayl-const})  and (\ref{xtayl-const-M}), we therefore have
\begin{equation}
\label{tau-bound}
\Big\vert \tau_{i+\frac{1}{2}}\Big\vert  \leq C_3 
\end{equation}
 with $C_3=\frac{\beta}{2}\big(2+\mathbf{M}_1\big) x_{\max}$.  In another hand we have:
\begin{eqnarray}
\label{inv-trans-expr1}
\frac{1}{\tau_{i+\frac{1}{2}}}  &  =   &  \frac{l_{i+1}k_i+l_ik_{i+1}}{k_ik_{i+1}} \nonumber\\
%&      & \nonumber\\
&  = &   \frac{l_{i+1}\times\frac{\sigma^2}{6l_i}\Big(x_{i+\frac{1}{2}}^3-x_{i-\frac{1}{2}}^3\Big)
	+l_i\times\frac{\sigma^2}{6l_{i+1}}\Big(x_{i+\frac{3}{2}}^3-x_{i+\frac{1}{2}}^3\Big)}{\frac{\sigma^2}{6l_i}\Big(x_{i+\frac{1}{2}}^3-x_{i-\frac{1}{2}}^3\Big)\cdot \frac{\sigma^2}{6l_{i+1}}\Big(x_{i+\frac{3}{2}}^3-x_{i+\frac{1}{2}}^3\Big)} \nonumber\\
%&     & \nonumber\\
&  =   & \frac{36l_il_{i+1}\Bigg(l_{i+1}\times\frac{\sigma^2}{6l_i}\Big(x_{i+\frac{1}{2}}^3-x_{i-\frac{1}{2}}^3\Big)
	+l_i\times\frac{\sigma^2}{6h_{i+1}}\Big(x_{i+\frac{3}{2}}^3-x_{i+\frac{1}{2}}^3\Big)\Bigg)}{\sigma^4\Big(x_{i+\frac{1}{2}}^3-x_{i-\frac{1}{2}}^3\Big)\Big(x_{i+\frac{3}{2}}^3-x_{i+\frac{1}{2}}^3\Big)}\nonumber\\
%&     & \nonumber\\
&   =  & \frac{6\sigma^2l_{i+1}^2\Big(x_{i+\frac{1}{2}}^3-x_{i-\frac{1}{2}}^3\Big)+6\sigma^2l_{i}^2\Big(x_{i+\frac{3}{2}}^3-x_{i+\frac{1}{2}}^3\Big)}{\sigma^4\Big(x_{i+\frac{1}{2}}^3-x_{i-\frac{1}{2}}^3\Big)\Big(x_{i+\frac{3}{2}}^3-x_{i+\frac{1}{2}}^3\Big)} \nonumber\\
%&    & \nonumber\\
&  = & \frac{6l_{i+1}^2}{\sigma^2x_{i+\frac{3}{2}}^3\Bigg(1-\Big(\frac{x_{i+\frac{1}{2}}}{x_{i+\frac{3}{2}}}\Big)^3\Bigg)}+\frac{6l_{i}^2}{\sigma^2x_{i+\frac{1}{2}}^3\Bigg(1-\Big(\frac{x_{i-\frac{1}{2}}}{x_{i+\frac{1}{2}}}\Big)^3\Bigg)} \nonumber\\
%&    & \nonumber\\
& =  & \frac{6}{\sigma^2x_{i+\frac{3}{2}}^2\Bigg(1-\Big(\frac{x_{i+\frac{1}{2}}}{x_{i+\frac{3}{2}}}\Big)^3\Bigg)} \times \frac{l_{i+1}^2}{x_{i+\frac{3}{2}}}+\frac{6}{\sigma^2x_{i+\frac{1}{2}}^2\Bigg(1-\Big(\frac{x_{i-\frac{1}{2}}}{x_{i+\frac{1}{2}}}\Big)^3\Bigg)}  \times \frac{l_i^2}{x_{i+\frac{1}{2}}}\nonumber\\
%&     & \nonumber\\
& \leq &  \dfrac{6}{\sigma^2x_{\max}^2} \cdot \frac{1}{\frac{x_{i+\frac{1}{2}}^2}{x_{\max}^2}}\Bigg[ \frac{1}{1-\Bigg(\frac{x_{i+\frac{1}{2}}}{x_{i+\frac{3}{2}}}\Bigg)^3}\times \frac{l_{i+1}^2}{x_{i+\frac{3}{2}}}
+\frac{1}{1-\Bigg(\frac{x_{i-\frac{1}{2}}}{x_{i+\frac{1}{2}}}\Bigg)^3}\times \frac{l_i^2}{x_{i+\frac{1}{2}}}\Bigg].
\end{eqnarray}
%\begin{eqnarray}
%&     & \nonumber\\
%&     & \nonumber\\   
%\frac{1}{\tau_{i+\frac{1}{2}}}                               &  \leq   & \frac{6}{\sigma^2x_{\max}} \cdot \frac{1}{\frac{x_{i+\frac{1}{2}}}{x_{\max}}}\Bigg[ \frac{1}{1-\Bigg(\frac{x_{i+\frac{1}{2}}}{x_{i+\frac{3}{2}}}\Bigg)^3}
%+\frac{1}{1-\Bigg(\frac{x_{i-\frac{1}{2}}}{x_{i+\frac{1}{2}}}\Bigg)^3}\Bigg]
%\end{eqnarray}	
Moreover, we have
\begin{eqnarray*}
	l_i=x_{i+\frac{1}{2}}-x_{i-\frac{1}{2}}~~~then~~~l_i \leq x_{i+\frac{1}{2}}~~~& thus &~~~\frac{1}{x_{i+\frac{1}{2}}} \leq \frac{1}{l_i}\\
	%&    & \\
	& similarly &  \frac{1}{x_{i+\frac{3}{2}}} \leq \frac{1}{l_{i+1}}.
\end{eqnarray*}
Thereby, we have
\begin{eqnarray}
\label{hbound}
\frac{l_i^2}{x_{i+\frac{1}{2}}} \leq l_i~~~~~and~~~~\frac{l_{i+1}^2}{x_{i+\frac{3}{2}}} \leq l_{i+1}.
\end{eqnarray}
Besides, we set 
\begin{equation}
W_i=\frac{x_{i+\frac{1}{2}}}{x_{\max}}~~~~~~~~~~Y_i=\frac{x_{i+\frac{1}{2}}}{x_{i+\frac{3}{2}}}~~~~~~~~~Z_i=\frac{x_{i-\frac{1}{2}}}{x_{i+\frac{1}{2}}}.
\end{equation}
Coming back to  (\ref{inv-trans-expr1}) and using Assumption \ref{assum2}, equations (\ref{hbound}) and  (\ref{ass-sp}) we have 
\begin{eqnarray}
\label{inv-trans-1}
\frac{1}{\tau_{i+\frac{1}{2}}}   &  \leq   &  \dfrac{6}{\sigma^2x_{\max}^2} \cdot \frac{1}{W_i^2}\cdot  \Bigg[\frac{1}{1-Y_i^3}l_{i+1}+\frac{1}{1-Z_i^3}l_i\Bigg] \nonumber\\
&     & \nonumber\\
\frac{1}{\tau_{i+\frac{1}{2}}}  & \leq & \dfrac{6}{\sigma^2x_{\max}^2} \cdot \frac{1}{W_i^2} \cdot  \Bigg[\frac{c}{1-Y_i^3}+\frac{1}{1-Z_i^3}\Bigg] l_{i}.
\end{eqnarray}
Let us notice also  that for $i=0,\ldots,N$
\begin{eqnarray}
0 <W_i < 1~~~~~~~~~~~~~~~~~~0 <Y_i < 1~~~~~~~~~~~~~~~~~~~~~~~~0< Z_i< 1,
\end{eqnarray}
then
\begin{eqnarray}
0<W_i^2< 1~~~~~~~~~~~~~~~~~~0< Y_i^3<1~~~~~~~~~~~~~~~~~~0 < Z_i^3< 1.
\end{eqnarray}
Using the Taylor expansion, we have:
\begin{eqnarray}
\frac{1}{W_i^2}&  =   & 1+(1-W_i^2)+\mathcal{O}(W_i^4) \nonumber\\
%&    &  \nonumber\\
\frac{1}{1-Y_i^3}&  =   & 1+Y_i^3+\mathcal{O}(Y_i^6) \nonumber\\
%&    &  \\
\frac{1}{1-Z_i^3}&  =   & 1+Z_i^3+\mathcal{O}(Z_i^6). \nonumber
\end{eqnarray}
This leads to 
\begin{eqnarray}
\label{yztayl-const}
\frac{1}{W_i^2}&  \leq   & 1+(1-W_i^2)+\mathbf{M_2}W_i^2 \leq 2+\mathbf{M_2}  \nonumber\\
%&      &   \nonumber\\
\frac{1}{1-Y_i^3}&  \leq   & 1+Y_i^3+\mathbf{M_3}Y_i^6 \leq 2+\mathbf{M_3}  \\
%&    & \nonumber \\
\frac{1}{1-X_i}&  \leq   & 1+Z_i+ \mathbf{M_4}X_i^2 \leq 2+\mathbf{M_4},  \nonumber
\end{eqnarray}
where $\mathbf{M_2},\mathbf{M_3},\mathbf{M_4}$ are positive constants.\
Using (\ref{inv-trans-1}),and (\ref{yztayl-const}), we get
\begin{eqnarray*}
	\frac{1}{\tau_{i+\frac{1}{2}}} \leq   \frac{6}{\sigma^2x_{\max}^2}\cdot \Bigg(2+\mathbf{M_2}\Bigg)\Bigg(2c+2+\mathbf{M_3}+\mathbf{M_4}\Bigg) l_i,
\end{eqnarray*}
therefore we have
\begin{equation}
\label{inv-trans-bound}
\frac{1}{\tau_{i+\frac{1}{2}}} \leq C_4 h_i,~~~~~~~~~~~~~
\end{equation}
with    $l_i\leq \frac{1}{2}\big(1+c\big)h_i$ for $i=0,\ldots,N$ and 
\begin{equation*}
C_4=\frac{3(1+c)}{\underline{\sigma}^2x_{\max}^2}\cdot \Bigg(2+\mathbf{M_2}\Bigg)\Bigg(2c+2+\mathbf{M_3}+\mathbf{M_4}\Bigg).
\end{equation*}
 \end{proof}
 Let us prove now, proposition \eqref{flux-c}.
 \begin{proof} Here we have two cases which are:
\begin{itemize}
	\item \underline{$1^{st} case$}:The TPFA method  is applied for the spatial discretization.  Thereby, for $i=0,\ldots,N$ we have:
\begin{eqnarray*}
	\Big\vert F_h\big(I_h\omega(x_{i+\frac{1}{2}})\big)-F\big(\omega(x_{i+\frac{1}{2}})\big) \Big\vert  & = & \Bigg\vert -\tau_{i+\frac{1}{2}}\Big(\omega(x_{i+1})-\omega(x_i)\Big)-x_{i+\frac{1}{2}}\Big(b^+ \omega(x_i)+b^-\omega(x_{i+1})\Big) \\
	%&    & \\
	&    & +k(x_{i+\frac{1}{2}})\omega
	'(x_{i+\frac{1}{2}})+bx_{i+\frac{1}{2}}\omega\big(x_{i+\frac{1}{2}}\big)\Bigg\vert \\
%	&    & \\
%	&  = &  \Bigg\vert k(x_{i+\frac{1}{2}})\Bigg(\omega
%	'(x_{i+\frac{1}{2}})- \frac{\omega(x_{i+1})-\omega(x_i)}{h_i}  \Bigg)\\
%	&    & \\
%	&    &  + x_{i+\frac{1}{2}}  \Bigg(b\omega\big(x_{i+\frac{1}{2}}\big)-\Big(b^+ \omega(x_i)+b^-\omega(x_{i+1})\Big)\Bigg)\\
%	&    & \\
%	&    & + \Bigg(\frac{k(x_{i+\frac{1}{2}})}{h_i}-\tau_{i+\frac{1}{2}}\Bigg)\Big(\omega(x_{i+1})-\omega(x_i)\Big)\Bigg\vert  \\
%	&    & \\
	\Big\vert F_h\big(\omega(x_{i+\frac{1}{2}})\big)-F\big(\omega(x_{i+\frac{1}{2}})\big) \Big\vert &  \leq   &    \Big\vert k(x_{i+\frac{1}{2}}) \Big\vert \cdot \Bigg\vert \omega
	'(x_{i+\frac{1}{2}})- \frac{\omega(x_{i+1})-\omega(x_i)}{h_i} \Bigg\vert\\
	%&    & \\
	&    &  +  x_{i+\frac{1}{2}}  \Bigg\vert b\omega\big(x_{i+\frac{1}{2}}\big)-\Big(b^+ \omega(x_i)+b^-\omega(x_{i+1})\Big)\Bigg\vert  \\
	%&    & \\
	&    & +\Bigg\vert  \frac{k(x_{i+\frac{1}{2}})}{h_i}-\tau_{i+\frac{1}{2}} \Bigg\vert \cdot  \Bigg\vert \omega(x_{i+1})-\omega(x_i) \Bigg\vert.
\end{eqnarray*}	
Let us estimate $P_1$ defined as follows:
\begin{equation}
\label{P1}
P_1:=k\big(x_{i+\frac{1}{2}}\big)\Bigg(\frac{\omega(x_{i+1})-\omega(x_i)}{h_i}-\omega'\big(x_{i+\frac{1}{2}}\big)\Bigg).
\end{equation}
Indeed  using  the Sobolev embedding theorem, as we are in dimension 1, $H^2(\Omega) \hookrightarrow  C^1(\Omega)$, using the Taylor expansion with integral remainder,  we have
%\begin{equation*}
%\omega(x_{i+1})=\omega\big(x_{i+\frac{1}{2}}\big)+\big(x_{i+1}-x_{i+\frac{1}{2}}\big)\omega'\big(x_{i+\frac{1}{2}}\big)
%+\int_{x_{i+\frac{1}{2}}}^{x_{i+1}}\big(x_{i+1}-x\big)\omega''(x)dx
%\end{equation*}
\begin{equation}
\label{tayexpi+1}
\omega(x_{i+1})=\omega\big(x_{i+\frac{1}{2}}\big)+\frac{h_i}{2}\omega'\big(x_{i+\frac{1}{2}}\big)+\int_{x_{i+\frac{1}{2}}}^{x_{i+1}}\big(x_{i+1}-x\big)\omega''(x)dx.
\end{equation}
Similarly, we have
\begin{equation}
\label{tayexpi}
\omega(x_{i})=\omega\big(x_{i+\frac{1}{2}}\big)-\frac{h_i}{2}\omega'\big(x_{i+\frac{1}{2}}\big)
+\int_{x_{i+\frac{1}{2}}}^{x_{i}}\big(x_{i}-x\big)\omega''(x)dx.
\end{equation}
Using (\ref{tayexpi+1}) and (\ref{tayexpi}), we get 
\begin{equation}
\frac{\omega(x_{i+1})-\omega(x_i)}{h_i}-\omega'(x_{i+\frac{1}{2}})=\frac{1}{h_i}\int_{x_{i}}^{x_{i+\frac{1}{2}}}\big(x_{i}-x\big)\omega''(x)dx+\frac{1}{h_i}\int_{x_{i+\frac{1}{2}}}^{x_{i+1}}\big(x_{i+1}-x\big)\omega''(x)dx.
\end{equation}
Since
\begin{eqnarray*}
	\Bigg\vert \int_{x_{i+\frac{1}{2}}}^{x_{i+1}}\big(x_{i+1}-x\big)\omega''(x)dx \Bigg\vert  &   \leq & \frac{h_i}{2} \int_{x_{i+\frac{1}{2}}}^{x_{i+1}} \vert \omega''  \vert  dx  \\
	%&    & \\
	%&    & \\
	\Bigg\vert \int_{x_{i}}^{x_{i+\frac{1}{2}}}\big(x_{i}-x\big)\omega''(x)dx \Bigg\vert  &   \leq & \frac{h_i}{2} \int_{x_{i}}^{x_{i+\frac{1}{2}}} \vert \omega''  \vert  dx ,  
\end{eqnarray*}
then 
\begin{eqnarray*}
	\Bigg\vert  \frac{\omega(x_{i+1})-\omega(x_i)}{h_i}-\omega'\big(x_{i+\frac{1}{2}}\big) \Bigg\vert  &   \leq & \frac{1}{2} \int_{x_i}^{x_{i+1}} \vert \omega'' \vert dx.
\end{eqnarray*}
Besides, we have 
\begin{equation*}
k(x_{i+\frac{1}{2}}) = \frac{1}{2}\sigma^2x_{i+\frac{1}{2}}^2,
\end{equation*}
%and
%
%\begin{equation*}
%k_i \leq \frac{\sigma^2}{2} x_{\max}^2~~~~~~~~~for~i=1,\ldots,N
%\end{equation*}
thus
\begin{eqnarray}
\label{boundP1}
\vert P_1  \vert     \leq    \frac{\sigma^2 x_{i+\frac{1}{2}}^2}{4} \int_{x_i}^{x_{i+1}} \vert \omega'' \vert dx 
%&        & \nonumber\\
  \leq  \frac{\sigma^2}{4} \int_{x_i}^{x_{i+1}}  \Big(\frac{x_{i+\frac{1}{2}}}{x}\Big)^2 \vert  x^2\omega'' \vert dx.
\end{eqnarray}
%	\begin{Ass}{Local quasi-uniformity} of the spatial mesh.
%		There exists a constant $c>0$ such that
%		
%		\begin{equation}
%		\label{ass-sp}
%		\frac{l_{i+1}}{c} \leq l_i \leq cl_{i+1},~~~~~~~with~l_i=x_{i+1}-x_i
%		\end{equation}
%	\end{Ass}
For $x\in \Omega_i=[x_i;x_{i+1}]$, we have
\begin{eqnarray}
x_i \leq x \leq x_{i+1}    \Rightarrow  \frac{1}{x_{i+1}} \leq \frac{1}{x}   \leq  \frac{1}{x_i}  
%\nonumber\\
%&             & \nonumber \\
 \Rightarrow   \frac{x_{i+\frac{1}{2}}}{x_{i+1}} \leq \frac{x_{i+\frac{1}{2}}}{x}   \leq  \frac{x_{i+\frac{1}{2}}}{x_i}.
\end{eqnarray}
We have also, using (\ref{ass-sp})
\begin{eqnarray}
\dfrac{x_{i+\frac{1}{2}}}{x_i}    =   \frac{x_i+x_{i+1}}{2x_i}  
%&       & \nonumber\\
%&   =   & \frac{1}{2}\Big(1+\frac{x_{i+1}}{x_i}\Big) \nonumber\\
%&      &  \nonumber\\
%&  =   &   \frac{1}{2}\Bigg(1+\frac{x_i+h_i}{x_i}\Bigg) \nonumber\\
%&      & \nonumber\\
%&  = & \frac{1}{2}\Bigg(1+1+\frac{h_i}{x_i}\Bigg) \nonumber\\
%&     &  \nonumber\\
   =   \frac{1}{2}\Bigg(2+\frac{h_i}{x_{i-1}+h_{i-1}}\Bigg) \nonumber,
 \end{eqnarray}
   so
   \begin{eqnarray}
%&     & \nonumber\\
\dfrac{x_{i+\frac{1}{2}}}{x_i} \leq    \frac{1}{2}\Bigg(2+\frac{h_i}{h_{i-1}}\Bigg),
% \nonumber\\
%&    & \nonumber\\
 \leq  \frac{1}{2}\Bigg(2+c\Bigg)~~
%~~~~~~~~~~~using~~(\ref{ass-sp})  \nonumber\\
%&      & \nonumber\\
\dfrac{x_{i+\frac{1}{2}}}{x_i}  \leq  1+\frac{c}{2}. \nonumber\\
\end{eqnarray}
Coming back to (\ref{boundP1}), we have
\begin{equation}
\vert P_1  \vert    \leq  \frac{\sigma^2}{4} \Bigg( 1+\frac{c}{2}\Bigg)^2 \int_{x_i}^{x_{i+1}} \vert  x^2\omega'' \vert dx .
\end{equation}
Since, by definition of $F$, we have
\begin{equation}
F'\Big((\omega(x)\Big) = -a x^2\omega''(x)-(2a+b)x\omega'(x)-b\omega(x),
\end{equation}
then we have
\begin{equation}
a \vert x^2 \omega'' \vert  \leq   \vert F'\big(\omega(x)\big) \vert + \vert (2a+b)x \vert \cdot   \vert \omega'(x) \vert  + \vert b \vert \cdot \vert\ \omega(x) \vert.
\end{equation}
Thus, we obtain
\begin{equation}
\label{bound-P1}
\vert P_1  \vert    \leq C_{14} \int_{x_i}^{x_{i+1}} \Bigg(\vert F'\big(\omega(x)\big) \vert + \vert \omega'(x) \vert + \vert  \omega(x) \vert \Bigg)dx,
\end{equation}
with 
\begin{equation*}
C_{14}=\frac{\sigma^2}{4} \Bigg( 1+\frac{c}{2}\Bigg)^2 \max\Big(\big(\beta+ \bar{r}\big) x_{\max}, \bar{r}+\beta ,1\Big).
\end{equation*}
We have 
\begin{equation}
P_2:= x_{i+\frac{1}{2}}  \Bigg(b\omega\big(x_{i+\frac{1}{2}}\big)-\Big(b^+ \omega(x_i)+b^-\omega(x_{i+1})\Big)\Bigg).
\end{equation}	
\begin{enumerate}
	\item $\underline{1^{st} case}:b>0$: Let us estimate $P_2$ defined as follows:
	\begin{equation}
	\label{P2}
	P_2=bx_{i+\frac{1}{2}}\Big(\omega\big(x_{i+\frac{1}{2}}\big)-\omega(x_i)\Big).
	\end{equation}
	By applying the Taylor theorem with integral remainder we have also 
	\begin{equation}
	\omega(x_{i+\frac{1}{2}}) = \omega(x_i) + \int_{x_i}^{x_{i+\frac{1}{2}}} \omega'(x) dx,
	\end{equation}
	then
    \begin{equation}
	\label{boundP21}
	\vert P_2 \vert  \leq \vert b \vert  x_{\max} \int_{x_i}^{x_{i+1}}   \vert \omega'(x) \vert  dx ,
	\end{equation}
	\item $\underline{1^{st} case}:b<0$:
	\begin{equation}
	P_2=bx_{i+\frac{1}{2}}\Big(\omega\big(x_{i+\frac{1}{2}}\big)-\omega(x_{i+1})\Big).
	\end{equation}
	By applying the Taylor theorem with integral remainder we have also 
	\begin{equation}
	\omega(x_{i+1}) = \omega(x_{i+\frac{1}{2}}) + \int_{x_i}^{x_{i+\frac{1}{2}}} \omega'(x) dx,
	\end{equation}
	then
	\begin{equation}
	\label{boundP22}
	\vert P_2 \vert  \leq \vert b \vert  x_{\max} \int_{x_i}^{x_{i+1}}   \vert \omega'(x) \vert  dx,
	\end{equation}
	From (\ref{boundP21}) and (\ref{boundP22}), we have finally 
	\begin{equation}
	\label{boundP2}
	\vert P_2 \vert  \leq \Big(\bar{r}+\beta\Big)  x_{\max} \int_{x_i}^{x_{i+1}}   \vert \omega'(x) \vert  dx.
	\end{equation}	
\end{enumerate}	
Similarly, we have
\begin{equation}
\label{omeplus}
\omega(x_{i+1}) = \omega(x_i) + \int_{x_i}^{x_{i+1}} \omega'(x) dx.
\end{equation}
then
\begin{equation}
\label{omep}
\vert  \omega(x_{i+1})-\omega(x_i) \vert \leq \int_{x_i}^{x_{i+1}}   \vert \omega'(x) \vert dx.
\end{equation}
Let us bound 
\begin{equation}
\Big\vert \tau_{i+\frac{1}{2}} - \frac{k(x_{i+\frac{1}{2}})}{h_i}\Big\vert.
\end{equation}
Indeed  we have:
\begin{eqnarray}
\label{consts}
\Big\vert \tau_{i+\frac{1}{2}} - \frac{k(x_{i+\frac{1}{2}})}{h_i}\Big\vert &  \leq  &  \Big\vert \tau_{i+\frac{1}{2}}\Big\vert + \Big\vert \frac{k(x_{i+\frac{1}{2}})}{h_i}\Big\vert.
\end{eqnarray}
Here, we have:
\begin{eqnarray}
\Big\vert \frac{k(x_{i+\frac{1}{2}})}{h_i}\Big\vert & = & \frac{\sigma^2 x_{i+\frac{1}{2}}^2}{2h_i} 
%&    & \nonumber\\
%&  =  & \frac{\sigma^2}{2} \times \frac{x_{i+\frac{1}{2}}^2}{x_{i+1}-x_{i}} \nonumber \\
%&    & \nonumber\\
  =   \frac{\sigma^2}{4} \frac{(x_i+x_{i+1})^2}{x_{i+1}-x_{i}}
%\nonumber\\
%&    & \nonumber\\
%&  =  & \frac{\sigma^2}{4} \frac{x_{i+1}^2+2x_ix_{i+1}+x_i^2}{x_{i+1}-x_{i}}\nonumber\\
%&    & \nonumber \\
  =    \frac{\sigma^2}{4} \frac{x_{i+1}^2\Big(1+2\frac{x_i}{x_{i+1}}+\big(\frac{x_i}{x_{i+1}}\big)^2\Big)}{x_{i+1}\Big(1-\frac{x_i}{x_{i+1}}\Big)} 
%&   & \nonumber\\
  =   \frac{\sigma^2}{4}x_{i+1}\times \frac{1+2Z_i+Z_i^2}{1-Z_i} \nonumber\\
%&    & \nonumber\\
\Big\vert \frac{k(x_{i+\frac{1}{2}})}{h_i}\Big\vert	&  \leq  & \frac{\sigma^2}{4}x_{\max}\times \Big(1+2Z_i+Z_i^2\Big)\times \frac{1}{1-Z_i}.
\end{eqnarray}
with  $Z_i$ defined as follows
\begin{equation}
Z_i=\frac{x_i}{x_{i+1}}~~~~~~~~~~~~~~~~~0< Z_i < 1,
\end{equation}
using a similar argument as in (\ref{xtayl-const}) and (\ref{xtayl-const-M}), there exists a positive constant $\mathbf{M_4}$ such that 
\begin{equation}
\label{bound-kl}
\Big\vert \frac{k(x_{i+\frac{1}{2}})}{h_i}\Big\vert  \leq   \sigma^2x_{\max}\times \big(2+\mathbf{M_4}\big),
\end{equation}
in the other hand, using (\ref{tau-bound}) we have, 
\begin{equation}
%	\label{tau-bound}
\Big\vert \tau_{i+\frac{1}{2}}\Big\vert  \leq  \frac{\sigma^2}{2}\big(2+\mathbf{M}_1\big) x_{\max}.
\end{equation}
Coming back to (\ref{consts}), and by using (\ref{bound-kl}) and (\ref{tau-bound}) we have 
\begin{eqnarray}
\label{boundtauk}
\Big\vert \tau_{i+\frac{1}{2}} - \frac{k(x_{i+\frac{1}{2}})}{h_i}\Big\vert &  \leq  &  \Big\vert \tau_{i+\frac{1}{2}}\Big\vert + \Big\vert \frac{k(x_{i+\frac{1}{2}})}{h_i}\Big\vert  \nonumber\\
%&     & \nonumber\\
&  \leq  &  \frac{\sigma^2}{2}\big(2+\mathbf{M}_1\big) x_{\max} +  \sigma^2x_{\max}\Big(2+\mathbf{M_4}\Big) \nonumber\\
%&      &  \nonumber\\
%&      &  \nonumber\\
\Big\vert \tau_{i+\frac{1}{2}} - \frac{k(x_{i+\frac{1}{2}})}{h_i}\Big\vert      &  \leq  &   \sigma^2x_{\max}\Big(\mathbf{M_1}+\mathbf{M_4}+\frac{3}{2}\Big),
\end{eqnarray}
then the estimate of $P_3$  defined as follows
\begin{equation}
P_3:= \Bigg(\tau_{i+\frac{1}{2}}- \frac{k(x_{i+\frac{1}{2}})}{h_i}\Bigg) \Big(\omega_{i+1}-\omega_i\Big),
\end{equation}
is, by using  (\ref{omep}) and (\ref{boundtauk}). We  therefore have
\begin{eqnarray}
\label{boundP3}
\vert P_3 \vert &  \leq & \Bigg\vert \tau_{i+\frac{1}{2}}- \frac{k(x_{i+\frac{1}{2}})}{h_i} \Bigg\vert \Big\vert  \omega_{i+1}-\omega_i  \Big\vert \nonumber \\  
%&      & \nonumber\\
&   \leq &  \sigma^2x_{\max}\Big(\mathbf{M_1}+\mathbf{M_4}+\frac{3}{2}\Big)  \int_{x_i}^{x_{i+1}}   \vert \omega'(x) \vert dx  \nonumber  \\
%&     &  \nonumber \\
\vert P_3 \vert   &   \leq  & \sigma^2x_{\max}\Big(\mathbf{M_1}+\mathbf{M_4}+\frac{3}{2}\Big)\int_{x_i}^{x_{i+1}}   \vert \omega'(x) \vert dx.
\end{eqnarray}
Using (\ref{boundP1}),(\ref{boundP2}) and (\ref{boundP3})
\begin{eqnarray}
\label{flux-const-12}
\Big\vert F_h\big(w(x_{i+\frac{1}{2}})\big)-F\big(w(x_{i+\frac{1}{2}})\big) \Big\vert & \leq &   \vert P_1\vert + \vert P_2\vert + \vert P_3 \vert \nonumber \\
%&     & \nonumber \\
&  \leq & C_{14} \int_{x_i}^{x_{i+1}} \Bigg(\vert F'(\omega) \vert + \vert \omega' \vert + \vert  \omega \vert \Bigg)dx+  \vert r-\sigma^2 \vert  x_{\max} \int_{x_i}^{x_{i+1}}   \vert \omega'(x) \vert  dx \nonumber \\
%&    &  \nonumber\\
&    & + \sigma^2x_{\max}\Big(\mathbf{M_1}+\mathbf{M_4}+\frac{3}{2}\Big) \int_{x_i}^{x_{i+1}}   \vert \omega'(x) \vert dx  \nonumber \\
%&     &  \nonumber \\
\Big\vert F_h\big(w(x_{i+\frac{1}{2}})\big)-F\big(w(x_{i+\frac{1}{2}})\big) \Big\vert	   &   \leq & C_{1}\int_{x_i}^{x_{i+1}} \Bigg(\big\vert F'(\omega) \big\vert + \big\vert \omega' \big\vert+ \big\vert  \omega \big\vert \Bigg)dx,
\end{eqnarray}
with
\begin{equation*}
C_{1}= C_{14}+ \big(\bar{r}+\beta\big) x_{\max}+ \beta x_{\max}\Big(\mathbf{M_1}+\mathbf{M_4}+\frac{3}{2}\Big),
\end{equation*}
%Therefore, using  (\ref{flux-const-12}) and $\sqrt{h_i} \leq \sqrt{h}$, the flux consistency is
%\begin{eqnarray}
%\label{flux-const}
%\Big\vert F\big(w(x_{i+\frac{1}{2}})\big)-F_h\big(I_hw(x_{i+\frac{1}{2}})\big) \Big\vert                               &      \leq & C_1 \Bigg(\big\vert\big\vert F'(\omega) \big\vert\big\vert_{\mathbf{L^2(\Omega)}} +\big\vert\big\vert \omega' \big\vert\big\vert_{\mathbf{L^2(\Omega)}} +\big\vert\big\vert \omega \big\vert\big\vert_{\mathbf{L^2(\Omega)}}\Bigg) \sqrt{h} \nonumber\\
%\end{eqnarray}
%with 
%\begin{equation*}
%C_{1}= C_{14}+\big(\bar{r}+\beta\big) x_{\max}+ \beta x_{\max}\Big(\mathbf{M_1}+\mathbf{M_4}+\frac{3}{2}\Big)
%\end{equation*}
\item \underline{$2^{nd} case:$} The fitted TPFA method  is applied for space discretization 
 at $x_{1/2}$ we have:
\begin{eqnarray}
\Big\vert G_h(\omega\big(x_{1/2}\big)-F(\omega\big(x_{1/2}\big)\Big\vert  &  =  &  \Big\vert -\frac{1}{4}x_1(a+b)\omega(x_1) +k(x_{\frac{1}{2}})\omega'\big(x_{\frac{1}{2}}\big)+b x_{\frac{1}{2}} \omega\big(x_{\frac{1}{2}}\big) \Big\vert \nonumber \\
%&   & \nonumber\\
&  =  &  \Big\vert -\frac{1}{4}x_1(a+b)\Big(\omega(x_1)-\omega(x_0)\Big) +k(x_{\frac{1}{2}})\omega'\big(x_{\frac{1}{2}}\big)+b x_{\frac{1}{2}} \omega\big(x_{\frac{1}{2}}\big) \Big\vert \nonumber\\
%&   & \nonumber\\
&  = & \Bigg\vert k(x_{\frac{1}{2}}) \Bigg(\omega'\big(x_{\frac{1}{2}}\big)-\dfrac{\omega(x_{1})-\omega(x_0)}{h_0}\Bigg)+b x_{\frac{1}{2}}\Big(\omega(x_{1/2})-\omega(x_0)\Big)\nonumber\\
%&    & \nonumber\\
&   & + \Big(k(x_{1/2})-\frac{1}{4}x_1(a+b)\Big)\Big(\omega(x_1)-\omega(x_0)\Big) \Bigg\vert \nonumber\\
%&   & \nonumber\\
\Big\vert G_h(\omega\big(x_{1/2}\big)-F(\omega\big(x_{1/2}\big)\Big\vert & \leq  & \big\vert k(x_{\frac{1}{2}}) \big\vert \cdot  \Bigg\vert \omega'\big(x_{\frac{1}{2}}\big)-\dfrac{\omega(x_{1})-\omega(x_0)}{h_0}\Bigg\vert + \vert b \vert  x_{\frac{1}{2}}\Big\vert \omega(x_{1/2})-\omega(x_0)\Big\vert \nonumber \\
%&    & \nonumber\\
&    & \Big\vert k(x_{1/2})-\frac{1}{4}x_1(a+b)\Big\vert \Big\vert \omega(x_1)-\omega(x_0)\Big\vert,  \nonumber
\end{eqnarray}
it follows that:
\begin{eqnarray}
\Big\vert k(x_{1/2})-\frac{1}{4}x_1(a+b)\Big\vert & \leq &  \Big\vert k(x_{1/2})\Big\vert +\Big\vert\frac{1}{4}x_1(a+b)\Big\vert \nonumber\\
&   & \nonumber\\
& \leq  & \Big\vert k(x_{1/2})\Big\vert +\frac{1}{4}x_{\max}\big(\bar{r} +\beta\big). \nonumber 
\end{eqnarray}	
Thereby, using \eqref{bound-kl}, \ref{P1} and \ref{bound-P1}, \eqref{P2} and \ref{boundP21}, we get
\begin{equation}
\Big\vert G_h(\omega\big(x_{1/2}\big)-F(\omega\big(x_{1/2}\big)\Big\vert  \leq  C_2 \int_{x_i}^{x_{i+1}}  \Big(\big\vert F'(\omega) \big\vert +\big\vert \omega' \big\vert + \big\vert \omega \big\vert \Big)dx \nonumber\\
\end{equation}
where
\begin{equation*}
C_2= C_{14}+\Bigg[\big(\bar{r}+\beta\big)+  \Bigg(\beta\Big(\mathbf{M_1}+\frac{5}{2}\Big)+\frac{1}{4}\bar{r}\Bigg)\Bigg]x_{\max}.
\end{equation*}
Besides, since for $i=1,\ldots,N$
\begin{equation}
G_h\big(w(x_{i+\frac{1}{2}})\big)=F_h(w(x_{i+\frac{1}{2}})\big),
\end{equation}	
similarly to \ref{flux-const-12}, we get
\begin{eqnarray}
\label{flux-const}
\Big\vert F\big(w(x_{i+\frac{1}{2}})\big)-G_h\big(I_hw(x_{i+\frac{1}{2}})\big) \Big\vert                               &      \leq & C_2 \int_{x_i}^{x_{i+1}}  \Big(\big\vert F'(\omega) \big\vert +\big\vert \omega' \big\vert + \big\vert \omega \big\vert \Big)dx \nonumber
\end{eqnarray}
with 
\begin{equation*}
C_{2}= C_{14}+\big(\bar{r}+\beta\big)x_{\max}+ \beta x_{\max}\Big(\mathbf{M_1}+\mathbf{M_4}+\frac{3}{2}\Big).
\end{equation*}
Finally, when the fitted TPFA method is applied for the space discretization, there exists a positive constant satisfying   \eqref{fluxconst-fitTPFA}.
\end{itemize}
\end{proof}
\section{Full discretization and errors estimates}
\label{ful-dis}
Let  $0:=t_0<t_1<\ldots<t_{M-1}<t_M:=T$ be a subdivision of the time interval $[0,T]$ with the step sizes $\Delta t_m=t_{m+1}-t_m,~~m\in\{0,\ldots,M-1\}$ and $\Delta t= \max_{1\leq m \leq M-1} \Delta t_m$.
The full discretization of \eqref{self-adj} using the combination of the TPFA method  with the parameter $\theta\in [0,1]$  can be formulated as :
Find a sequence $u_h^1,\ldots,u_h^M\in V_h$ such that for $m\in\{0,\ldots,M-1\}$
\begin{equation}
\label{full-disc-TPFA}
\left\lbrace \begin{array}{lcr}
\Bigg(\dfrac{u^{m+1}_h-u^m_h}{\Delta t_m},v_h\Bigg)_h+a_h\Big(\theta u^{m+1}_h+(1-\theta)u^m_h,v_h;t_{m+\theta}\Big) & = &  \Bigg(\theta f^{m+1}+(1-\theta)f^m,v_h\Bigg)_h\\
~~~~~~~~~~~~~~~~~~~~~~~~~~~~~~~~~~~~~~~~~~~~~~~~~~~~~~~~~~~~~~~~~~~u^0 & = u_{oh} & 
\end{array} 
\right.	
\end{equation}
where $t_{m+\theta}=\theta t_{m+1}+(1-\theta)t_m$  and the bilinear form $a_h$ is given by \eqref{bilaflux-TPFA}.
Similarly, when the fitted TPFA method is applied for the the spatial discretization, the full discretization is formulated as follows:\\
Find a sequence $ u_h^1,\ldots,u_h^M \in V_h$ such that for $m\in\{0,\ldots,M-1\}$
\begin{equation}
\label{full-disc-fitTPFA}
\left\lbrace \begin{array}{lcr}
\Bigg(\dfrac{u^{m+1}_h-u^m_h}{\Delta t_m},v_h\Bigg)_h+b_h\Big(\theta u^{m+1}_h+(1-\theta)u^m_h,v_h;t_{m+\theta}\Big) & = &  \Bigg(\theta f^{m+1}+(1-\theta)f^m,v_h\Bigg)_h\\
~~~~~~~~~~~~~~~~~~~~~~~~~~~~~~~~~~~~~~~~~~~~~~~~~~~~~~~~~~~~~~~~~~~u^0 & = u_{oh} & 
\end{array} 
\right.	
\end{equation}
where the bilinear form $b_h$ is given by \eqref{bilbflux-fitTPFA}.  
\subsection{Errors estimates}
\begin{theo}
\label{maintheo}
Let us consider the unique solution $u$ of  \eqref{weak-form} and $\zeta_h^m$ the numerical solution  of the fully discretized scheme using the TPFA method \eqref{form-TPFA}($\zeta_h^m=u_h^m$ for the TPFA method) or the fitted TPFA method \eqref{form-fitTPFA}  ($\zeta_h^m=z_h^m$ for  fitted TPFA method). Let $\theta \in [1/2;1]$, 
	% with $I_h \zeta(0)=\zeta_h^0$.\\
	%If $u\in L^2\Big(0,T;H_0^2(\Omega)\Big)$
	if $u\in H^1\Big(0,T;H^1(\Omega)\Big) \cap H^2\Big(0,T; L^2(\Omega)\Big)$ and  $F(u) \in C(0,T,H^1(\Omega))$,
  then there exists a positive constant  $C$, independent of $h$, $\Delta t$, $M$, and $ N$ such that 
	\begin{equation}
	\big\vert\big\vert u(t_m)-\zeta_h^m\big\vert\big\vert_{0,h} \leq C(h+\Delta t).
	\end{equation}
\end{theo}
\begin{proof}
Indeed the proofs follow the same lines  as that  in \cite[Theorem 7]{angermann2007convergence}. We summarise the  keys steps.
Here we have two cases:\\
%\begin{itemize}
%\item 
\underline{$1^{st}$ case}  When the TPFA method is applied for the space discretization.\\
	Here, we take $\zeta_h^m=u_h^m$. Let us notice that 
	\begin{equation}
	\label{err}
	\big\vert\big\vert u(t_m)-u_h^m\big\vert\big\vert_{0,h} \leq  \big\vert\big\vert u(t_m)-I_hu(t_m)\big\vert\big\vert_{0,h} + \big\vert\big\vert I_hu(t_m)-u_h^m\big\vert\big\vert_{0,h}
	\end{equation}
	where $I_h$ is the interpolation operator introduced in \eqref{interpol-op}.
	In one hand, in order to bound the first term on the right hand side of \eqref{err}, let us recall the following result. Since $u(t)\in H^2(\Omega)$ then there exists a constant $C_{31}$ depending on $u$ (see Theorem 3.25, page 138 in \cite{KnaberAngermann2002}) such that
	\begin{equation}
	\big\vert\big\vert I_hu(t)-u(t)\big\vert\big\vert_{0,h} \leq C_{31} \cdot h^2 \cdot \vert u(t)\vert_2
	\end{equation}
	where $\vert \cdot \vert_2$  is the semi-norm of $H^2(\Omega)$. Furthermore, for $u\in C(\big(0,T\big),H^2(\Omega))$, there exists a positive constant $C_{32}=C_{31}(u,T)\cdot x_{\max}$ such that
	\begin{equation}
	\label{interp-err}
	\big\vert\big\vert I_hu(t_m)-u(t_m)\big\vert\big\vert_{0,h} \leq C_{32} \cdot h. 
	\end{equation}
	 We now estimate $W^m:=I_hu(t_m)-u_h^m$ in the discrete $L^2$-norm.
	We define the expression $A$ by 
	\begin{equation}
	A= \Bigg( \dfrac{W^{m+1}-W^m}{\Delta t_m},v_h \Bigg)_h +a_h\Big(\theta W^{m+1}+(1-\theta)W^m,v_h;t_{m+\theta}\Big),
	\end{equation}
	where $a_h$ is the bilinear form given by \eqref{bila} when the TPFA method is applied.
	 By some arithmetic manipulations, we have
	\begin{eqnarray}
	\label{est-A}
	A & = & \sum_{i=1}^N l_i \frac{W^{m+1}_i-W^m_i}{\Delta t_m}v_i+a_h\Big(\theta W^{m+1}+(1-\theta)W^m,v_h;t_{m+\theta}\Big) \nonumber\\
	%&    & \nonumber\\
	&  =  & \sum_{i=1}^N l_i\frac{I_hu_i(t_{m+1})-I_hu_i(t_m)}{\Delta t_m}v_i+a_h\Big(\theta I_hu(t_{m+1})+(1-\theta)I_hu(t_m),v_h;t_{m+\theta}\Big) \nonumber\\
	%&    & \nonumber\\
	&    & -\sum_{i=1}^N l_i\frac{{u_h}_i^{m+1}-{u_h}_i^m}{\Delta t_m}v_i-a_h\Big(\theta u_h^{m+1}+(1-\theta)u_h^m,v_h;t_{m+\theta}\Big) \nonumber\\
	%&    & \nonumber\\
	&  =  & \sum_{i=1}^N l_i\frac{I_hu_i(t_{m+1})-I_hu_i(t_m)}{\Delta t_m}v_i+a_h\Big(\theta I_hu(t_{m+1})+(1-\theta)I_hu(t_m),v_h;t_{m+\theta}\Big) \nonumber\\
	%&    & \nonumber\\
	&   & -\sum_{i=1}^N l_i\frac{{u_h}_i^{m+1}-{u_h}_i^m}{\Delta t_m}v_i-a_h\Big(\theta u_h^{m+1}+(1-\theta)u_h^m,v_h;t_{m+\theta}\Big)-\Bigg(\theta \dot{u}(t_{m+1})+\big(1-\theta\big)\dot{u}(t_m),L_hv_h\Bigg) \nonumber\\
	%&    &  \nonumber \\
	&   & +\Bigg(\theta \dot{u}(t_{m+1})+\big(1-\theta\big)\dot{u}(t_m),L_hv_h\Bigg)+\theta \hat{a}_h\big(u(t_{m+1}),v_h;t_{m+1}\big)+\big(1-\theta\big)\hat{a}_h\big(u(t_m),v_h;t_m\big) \nonumber \\
	%&    & \nonumber\\
	&    &  -\theta \hat{a}_h\big(u(t_{m+1}),v_h;t_{m+1}\big)-\big(1-\theta\big)\hat{a}_h\big(u(t_m),v_h;t_m\big).
	\end{eqnarray}
	We also have
	\begin{eqnarray}
	A &  =  & \Bigg[ \sum_{i=1}^N l_i\frac{I_hu_i(t_{m+1})-I_hu_i(t_m)}{\Delta t_m}v_i -\Bigg(\theta \dot{u}(t_{m+1})+\big(1-\theta\big)\dot{u}(t_m),L_hv_h\Bigg)\Bigg] \nonumber \\
	%&    & \nonumber\\
	&    & + \Bigg[a_h\Big(\theta I_hu(t_{m+1})+(1-\theta)I_hu(t_m),v_h;t_{m+\theta}\Big)-\theta \hat{a}_h\big(u(t_{m+1}),v_h;t_{m+1}\big)-\big(1-\theta\big)\hat{a}_h\big(u(t_m),v_h;t_m\big)\Bigg] \nonumber\\
	%&   & \nonumber\\
	&   & +\Bigg[\Bigg(\theta \dot{u}(t_{m+1})+\big(1-\theta\big)\dot{u}(t_m),L_hv_h\Bigg)+\theta \hat{a}_h\big(u(t_{m+1}),v_h;t_{m+1}\big) +\big(1-\theta\big)\hat{a}_h\big(u(t_m),v_h;t_m\big)\Bigg] \nonumber \\
	%&    & \nonumber\\
	&    & -\sum_{i=1}^N l_i\frac{{u_h}_i^{m+1}-{u_h}_i^m}{\Delta t_m}v_i-a_h\Big(\theta u_h^{m+1}+(1-\theta)u_h^m,v_h;t_{m+\theta}\Big).
	\end{eqnarray}
	%Besides, using (\ref{bilaflux-TPFA}), we have
	Remember that (see \eqref{var-form-TPFA})
	\begin{equation}
	\sum_{i=1}^N l_i\frac{{u_h}_i^{m+1}-{u_h}_i^m}{\Delta t_m}v_i+a_h\Big(\theta u_h^{m+1}+(1-\theta)u_h^m,v_h;t_{m+\theta}\Big)= \Bigg(\theta f^{m+1}+(1-\theta)f^m,v_h\Bigg)_h,
	\end{equation}
	and also
	\begin{eqnarray}
	\Bigg(\theta \dot{u}(t_{m+1})+\big(1-\theta\big)\dot{u}(t_m),L_hv_h\Bigg)+\theta \hat{a}_h\big(u(t_{m+1}),v_h;t_{m+1}\big) & +  & \nonumber\\
	\big(1-\theta\big)\hat{a}_h\big(u(t_m),v_h;t_m\big) &= &  \Bigg(\theta f^{m+1}+(1-\theta)f^m,L_hv\Bigg). \nonumber \\
	\end{eqnarray}
	Thereby, we get
	\begin{equation}
	A=Y_1^m+Y_2^m+Y^3_m,
	\end{equation}
	where
	\begin{eqnarray}
	Y_1^m=\sum_{i=1}^N l_i \frac{I_hu_i(t_{m+1})-I_hu_i(t_m)}{\Delta t_m}v_i-\Bigg(\theta \dot{u}(t_{m+1})+(1-\theta)\dot{u}(t_m),v_h\Bigg),
	\end{eqnarray}
	\begin{eqnarray}
	Y_2^m=a_h\Big(\theta I_hu(t_{m+1})+(1-\theta)I_hu(t_m),v_h;t_{m+\theta}\Big) - \theta \hat{a}_h\Big(u(t_{m+1}),v_h;t_{m+1}\Big) &   &  \nonumber\\
	-(1-\theta)\hat{a}_h\Big(u(t_m),v_h;t_m\Big)  &  & 
	\end{eqnarray}
	and 
	\begin{eqnarray}
	Y_3^m= \Bigg(\theta f^{m+1}+(1-\theta)f^m,L_hv\Bigg)- \Bigg(\theta f^{m+1}+(1-\theta)f^m,v_h\Bigg)_h.
	\end{eqnarray}
	The estimation  of  $Y_1^m$ is  done exactly as in \cite[(54)]{angermann2007convergence} and we have:
	 \begin{equation} 
		Y_1^m=\Big(\omega^m,L_hv_h\Big)
		\end{equation}
		\begin{equation}
		\omega^m:=\frac{L_hu(t_{m+1})-L_hu(t_m)}{\Delta t_m}-\theta \dot{u}(t_{m+1})-(1-\theta)\dot{u}(t_m)
		\end{equation}
\begin{equation}
		\label{Y1-est}
		|Y_1^m| \leq || \omega^m ||_{L^2(\Omega)}  ||v_h||_{0,h}
		\end{equation}
		with
		\begin{equation*}
		||\omega^m||_{L^2(\Omega)}  \leq \mathcal{Q}_1^m\big(\Delta t_m,h\big):= \frac{1}{\Delta t_m} \int_{t_m}^{t_{m+1}} \big\|(L_h-I\big)\circ \dot{u}(s)\big\|_{L^2(\Omega)} ds+ \int_{t_m}^{t_{m+1}} ||\ddot{u} ||_{L^2(\Omega)} ds. 
		\end{equation*}
	      Estimate of $Y_2^m$
		\begin{equation}
		Y_2^m=a_h\Big(\theta I_hu(t_{m+1})+(1-\theta)I_hu(t_m),v_h\Big)-\theta \hat{a}_h\Big(u(t_{m+1}),v_h\Big)-(1-\theta)\hat{a}_h\Big(u(t_m),v_h\Big).
		\end{equation}
		By adding and extracting the term $\hat{a}_h\Big(\theta u(t_{m+1})+(1-\theta)u(t_m),v_h,t_{m+\theta}\Big)$ we get
		\begin{eqnarray}
			\label{A-Y2}
			Y_2^m &  =  & a_h\Big(\theta I_hu(t_{m+1})+(1-\theta)I_hu(t_m),v_h,t_{m+\theta}\Big)-\hat{a}_h\Big(\theta u(t_{m+1})+(1-\theta)u(t_m),v_h,t_{m+\theta}\Big)\nonumber\\
			%&    & \nonumber\\
			&    & +\hat{a}_h\Big(\theta u(t_{m+1})+(1-\theta)u(t_m),v_h,t_{m+\theta}\Big)-\theta \hat{a}_h\Big(u(t_{m+1}),v_h,t_{m+1}\Big)\nonumber\\
			%&   & \nonumber\\
			&   & -(1-\theta)\hat{a}_h\Big(u(t_m),v_h,t_{m}\Big)\nonumber\\
			&  = :& Y_{21}^m+Y_{22}^m,
		\end{eqnarray}
		where
		%Let us denote $Y_{21}^m$ by
		\begin{equation}
		Y_{21}^m=a_h\Big(\theta I_hu(t_{m+1})+(1-\theta)I_hu(t_m),v_h;t_{m+\theta}\Big)-\hat{a}_h\Big(\theta u(t_{m+1})+(1-\theta)u(t_m),v_h;t_{m+\theta})\Big),
		\end{equation}
		and 
		%by $Y_{22}^m$ 
		\begin{equation}
		Y_{22}^m=\hat{a}_h\Big(\theta u(t_{m+1})+(1-\theta)u(t_m),v_h;t_{m+\theta})\Big)-\theta \hat{a}_h\Big(u(t_{m+1}),v_h;t_{m+1}\Big)-(1-\theta)\hat{a}_h\Big(u(t_m),v_h;t_{m}\Big),
		\end{equation}
		with $t_{m+\theta}=\theta t_{m+1}+(1-\theta)t_m~~~~~~\theta\in [1/2,1]$.
		%\begin{enumerate}
			%\item 
			Estimate of $Y_{21}^m$
			
%			\begin{eqnarray}
%			Y_{21}^m  & =  & a_h\Big(\theta I_hu(t_{m+1})+(1-\theta)I_hu(t_m),v_h;t_{m+\theta}\Big)-\hat{a}_h\Big(\theta u(t_{m+1})+(1-\theta)u(t_m),v_h;t_{m+\theta})\Big) \nonumber\\
%			\end{eqnarray}
			Note that 
			\begin{eqnarray}
			\label{Y21-exp}
			Y_{21}^m  & = & \theta \Big(a_h\big(I_hu(t_{m+1}),v_h;t_{m+\theta}\big)-\hat{a}_h\big(u(t_{m+1}),v_h;t_{m+\theta}\big)\Big) \nonumber\\
			%&    &  \nonumber\\
			&    & +  (1-\theta) \Big(a_h\big(I_hu(t_{m}),v_h;t_{m+\theta}\big)-\hat{a}_h\big(u(t_{m}),v_h;t_{m+\theta}\big)\Big)\nonumber.
			\end{eqnarray}
			Let us consider the term
			\begin{equation}
			\delta_{21}(\omega,v_h,s):=a_h(I_h\omega,v_h;s)-\hat{a}_h(\omega,v_h;s)
			\end{equation}
			Thereby, using \eqref{bilaflux-TPFA} we have:
			\begin{eqnarray}
			\label{bil-conv}
			\delta_{21}(\omega,v_h,s)  & =  &  a_h(I_h \omega,v_h;s)-\hat{a}_h(\omega,v_h;s)  \nonumber \\
			%&    &  \nonumber \\
			&  = & \sum_{i=1}^N\Bigg(F_h\big(I_h \omega(x_{i+\frac{1}{2}})\big)-F_h\big(I_h \omega(x_{i-\frac{1}{2}})\big)\Bigg)\cdot v_i + \Big(c(s)I_h\omega,v_h\Big)_h  \nonumber\\
			%&    & \nonumber\\
			&    &  -\sum_{i=1}^N\Bigg(F\big(\omega(x_{i+\frac{1}{2}})\big)-F\big(\omega(x_{i-\frac{1}{2}})\big)\Bigg)\cdot v_i-\Big(c(s)\omega,L_hv_h\Big) \nonumber \\
			%&    & \nonumber \\
			&   =  &  \sum_{i=1}^N\Bigg(F_h\big(I_h \omega(x_{i+\frac{1}{2}})\big)-F\big(\omega(x_{i+\frac{1}{2}})\big)\Bigg)  v_i-\sum_{i=1}^N\Bigg(F_h\big(I_h \omega(x_{i-\frac{1}{2}})\big) \nonumber \\
			&    & -F\big(\omega(x_{i-\frac{1}{2}})\big)\Bigg) \cdot v_i+ c(s)\Bigg(\Big(I_h\omega,v_h\Big)_h-\Big(\omega,L_hv_h\Big)\Bigg) \nonumber\\
			%&    & \nonumber\\
			&  =  &\sum_{i=1}^N\Bigg(F_h\big(I_h \omega(x_{i+\frac{1}{2}})\big)-F\big(\omega(x_{i+\frac{1}{2}})\big)\Bigg) \cdot v_i-\sum_{i=0}^{N-1}\Bigg(F_h\big(I_h \omega(x_{i+\frac{1}{2}})\big)\nonumber\\
			%&    & \nonumber\\
			&    & -F\big(\omega(x_{i+\frac{1}{2}})\big)\Bigg) \cdot v_{i+1} +c(s)\Bigg(\Big(L_hI_h\omega,L_hv_h\Big)-\Big(\omega,L_hv_h\Big)\Bigg). \nonumber
			\end{eqnarray}
			We can also expand as
			\begin{eqnarray}
			\label{deltat}
			\delta_{21}(\omega,v_h,s)  
			& = &  -\Bigg(F_h\big(I_h\omega(x_{\frac{1}{2}})\big)-F\big(\omega(x_{\frac{1}{2}})\big)\Bigg) v_{1}+\sum_{i=1}^{N-1}\Bigg(F_h\big(I_h \omega(x_{i+\frac{1}{2}})\big)-F\big(\omega(x_{i+\frac{1}{2}})\big)\Bigg) \cdot \big(v_i-v_{i+1}\big)\nonumber\\
			%&    & \nonumber\\
			&    & +\Big(F_h\big(I_h \omega(x_{N+\frac{1}{2}})\big)-F\big(\omega(x_{N+\frac{1}{2}})\big)\Big) \cdot v_N + c(s)\Bigg(\Big(L_h\omega,L_hv_h\Big)-\Big(\omega,L_hv_h\Big)\Bigg) \nonumber\\
			%\end{eqnarray}
		%	\begin{eqnarray}
		%	\label{Y2-delta1}
			&  =  & \Bigg(F_h\big(I_h\omega(x_{\frac{1}{2}})\big)-F\big(\omega(x_{\frac{1}{2}})\big)\Bigg) \cdot   \Big(v_0-v_{1}\Big)  \nonumber\\
			%&    & \nonumber\\
			&    & +\sum_{i=1}^{N-1}\Bigg(F_h\big(I_h \omega(x_{i+\frac{1}{2}})\big)-F\big(\omega(x_{i+\frac{1}{2}})\big)\Bigg) \cdot \big(v_i-v_{i+1}\big)\nonumber\\
			%&    & \nonumber\\
			&    & +\Big(F_h\big(I_h \omega(x_{N+\frac{1}{2}})\big)-F\big(\omega(x_{N+\frac{1}{2}})\big)\Big) \cdot \Big(v_N-v_{N+1}\Big) + c(s)\Big(L_h\omega-w,L_hv_h\Big)   \nonumber\\
			%&    & \nonumber\\
			%\delta_{21}(\omega,v_h,s)   
			 &  =  & \sum_{i=0}^{N} \Bigg(F_h\big(I_h \omega(x_{i+\frac{1}{2}})\big)-F\big(\omega(x_{i+\frac{1}{2}})\big)\Bigg) \cdot \big(v_i-v_{i+1}\big) + c(s)\Big(\big(L_h-I\big)\omega,L_hv_h\Big), \nonumber\\
			 &  =  & \delta_{211}+\delta_{212},
			\end{eqnarray}	
			where  $I$ is the identity operator and $\delta_{211}$ defined as follows:
		\begin{equation}
		\label{delta11}
		\delta_{211} = \sum_{i=0}^{N} \Bigg(F_h\big(I_h \omega(x_{i+\frac{1}{2}})\big)-F\big(\omega(x_{i+\frac{1}{2}})\big)\Bigg) \cdot \big(v_{i}-v_{i+1}\big)
		\end{equation}
			Moreover, we  also have
            \begin{eqnarray*}
             \label{delta10}
            \vert F_h\big(I_h \omega(x_{\frac{3}{2}})\big)-F\big(\omega(x_{\frac{1}{2}})\big) \big(v_0-v_{1}\big) \vert &  = &  \vert \frac{1}{\tau_{1/2}} \Bigg(F_h\big(I_h \omega(x_{\frac{3}{2}})\big)-F\big(\omega(x_{\frac{1}{2}})\big)\Bigg) \tau_{1/2}\big(v_0-v_{1}\big) \vert \\
            %\nonumber\\
             &  \leq   & \frac{1}{\tau_{1/2}} \Bigg( F_h\big(I_h \omega(x_{\frac{3}{2}})\big)-F\big(\omega(x_{\frac{1}{2}})\big)\Bigg)  \tau_{1/2}  \vert v_{1} \vert \nonumber\\ 
             %&      & \nonumber\\
             &  \leq   & C_4 h_0  \int_{x_0}^{x_1} \Big(\vert F'(\omega)\vert + \vert \omega' \vert + \vert \omega \vert \Big) dx   \times  C_3 \vert v_1 \vert    \nonumber\\
             %&    & \nonumber\\
             & \leq  & C_{15} h_0 \Bigg[ \int_{x_0}^{x_1} \Big(\vert F'(\omega)\vert + \vert \omega' \vert + \vert \omega \vert \Big)^2 dx  \Bigg]^{1/2} \sqrt{h_0 v_1^2} \nonumber\\
             %&    & \nonumber\\
            \vert F_h\big(I_h \omega(x_{\frac{3}{2}})\big)-F\big(\omega(x_{\frac{1}{2}})\big) \big(v_0-v_{1}\big) \vert&  \leq  &  C_{15} h_0 \Bigg[ \int_{x_0}^{x_1} \Big(\vert F'(\omega)\vert + \vert \omega' \vert + \vert \omega \vert \Big)^2 dx  \Bigg]^{1/2} \sqrt{h_0 v_1^2},
             \end{eqnarray*}			
	where the positive constant $C_2$  and $C_3$ are given in \eqref{tau-bounds} and $C_{15}=\max\Big(C_3,C_4\Big)$.
Besides, we have: 			
			\begin{eqnarray}
			\label{bound-flux-tau}
			\sum_{i=1}^N   \Bigg(F_h\big(I_h \omega(x_{i+\frac{1}{2}})\big)-F\big(\omega(x_{i+\frac{1}{2}})\big)\Bigg)  \Big(v_i-v_{i+1}\Big) & = &  \sum_{i=1}^N \frac{1}{\sqrt{\tau_{i+\frac{1}{2}}}}\Bigg(F_h\big(I_h\omega(x_{i+\frac{1}{2}})\big)  \\
			&    & -F\big(\omega(x_{i-\frac{1}{2}})\big)\Bigg)   \times \sqrt{\tau_{i+\frac{1}{2}}} \Big(v_i-v_{i+1}\Big) \nonumber
			\end{eqnarray}
			Thereby, using  \eqref{fluxconst-TPFA} and  \eqref{tau-bounds}  in \eqref{bound-flux-tau} then we have
			\begin{eqnarray*}
			\sum_{i=1}^N  \vert \Bigg(F_h\big(I_h \omega(x_{i+\frac{1}{2}})\big)-F\big(\omega(x_{i+\frac{1}{2}})\big)\Bigg)  \big(v_i-v_{i+1}\big) \vert
			%& = &  \sum_{i=0}^N  \frac{1} {\sqrt{\tau_{i+\frac{1}{2}}}}\Bigg(F_h\big(I_h\omega(x_{i+\frac{2}})\big)  
			%&    & -F\big(\omega(x_{i+\frac{1}{2}})\big)\Bigg)\times \sqrt{\tau_{i+\frac{1}{2}}}\Bigg(v_i-v_{i+1}\Bigg)  \nonumber\\
			%&    & \nonumber\\
			& \leq &  \sum_{i=1}^N \Bigg(\sqrt{C_2 h_i}   \int_{x_i}^{x_{i+1}}  \Big(\big\vert F'(\omega) \big\vert +\big\vert \omega' \big\vert + \big\vert \omega \big\vert \Big)dx \nonumber\\
			%&    & \nonumber\\
			&    &   \times \sqrt{\tau_{i+\frac{1}{2}}}\vert \Big(v_i-v_{i+1}\Big)\vert \Bigg) \nonumber\\
			%&    & \nonumber\\
			%&    &   \nonumber\\
			%&    & \nonumber\\
			&   \leq  &  \sqrt{C_4} \sum_{i=1}^N \Bigg( h_i \Bigg[ \int_{x_i}^{x_{i+1}}  \Big(\big\vert F'(\omega) \big\vert +\big\vert \omega' \big\vert + \big\vert \omega \big\vert \Big)^2dx \Bigg]^{1/2}  \nonumber\\
			%&    & \nonumber\\
			&    &  \times  \sqrt{\tau_{i+\frac{1}{2}}} \vert \Big(v_i-v_{i+1}\Big) \vert \Bigg) \nonumber\\
			%&     & \nonumber\\
			&   \leq  &  \sqrt{C_4}  \sum_{i=1}^N  \Bigg( \Bigg[ \int_{x_i}^{x_{i+1}}  \Big(\big\vert F'(\omega) \big\vert +\big\vert \omega' \big\vert + \big\vert \omega \big\vert \Big)^2dx \Bigg]^{1/2}  \nonumber\\
			%&    & \nonumber\\
			&    &  \times  \sqrt{\tau_{i+\frac{1}{2}}} \vert \Big(v_i-v_{i+1}\Big)\vert \Bigg) h \nonumber
			\end{eqnarray*}
			Indeed, we also have 
		\begin{eqnarray}
\label{delta1i}
	\sum_{i=1}^N   \vert \Bigg(F_h\big(I_h \omega(x_{i+\frac{1}{2}})\big)-F\big(\omega(x_{i+\frac{1}{2}})\big)\Bigg)  \big(v_i-v_{i+1}\big)   \vert         &  \leq  & \sqrt{C_4} \Bigg( \sum_{i=1}^N \int_{x_i}^{x_{i+1}} \Big(\big\vert F'(\omega) \big\vert +\big\vert \omega' \big\vert + \big\vert \omega \big\vert \Big)^2dx \Bigg)^{1/2} \\
                       % &    & \nonumber\\
                        &    & \times  \Bigg(\sum_{i=1}^N \tau_{i+\frac{1}{2}}\Big(v_i-v_{i+1}\Big)^2 \Bigg)^{1/2} h \nonumber\\
                       % &    & \nonumber\\
 	\sum_{i=1}^N  \vert \Bigg(F_h\big(I_h \omega(x_{i+\frac{1}{2}})\big)-F\big(\omega(x_{i+\frac{1}{2}})\big)\Bigg)  \big(v_i-v_{i+1}\big) \vert                          & \leq  & h\sqrt{C_4} \Bigg[\int_{x_1}^{x_{N+1}} \Big(\big\vert F(\omega) \big\vert +\big\vert \omega' \big\vert + \big\vert \omega \big\vert \Big)^2dx \Bigg]^{1/2}  \big\vert \big\vert v_h  \big\vert\big\vert_{0,\omega}  \nonumber\\
\end{eqnarray}		
Coming back to $\delta_{11}$ defined in \eqref{delta11}, and using the equations \eqref{delta10}, \eqref{delta1i}, and also the fact that $h_0\leq c l_1$ from Assumption \eqref{assum2}, we get:
\begin{eqnarray}
\delta_{211} &  \leq &  C_{16}\Bigg(h_0 \Bigg[ \int_{x_0}^{x_1} \Big(\vert F'(\omega)\vert + \vert \omega' \vert + \vert \omega \vert \Big)^2 dx  \Bigg]^{1/2} \sqrt{h_0 v_1^2} \\
&&+\Bigg[\int_{x_1}^{x_{N+1}} \Big(\big\vert F'(\omega)\big\vert +\big\vert \omega' \big\vert + \big\vert \omega \big\vert \Big)^2dx \Bigg]^{1/2}  h  \big\vert \big\vert v_h  \big\vert\big\vert_{0,\omega} \Bigg)\nonumber\\
%&    & \nonumber\\
& \leq  &  C_{17} h\Bigg[\int_{x_0}^{x_{N+1}} \Big(\big\vert F' (\omega)\big\vert +\big\vert \omega' \big\vert + \big\vert \omega \big\vert \Big)^2dx \Bigg]^{1/2}   \Big(\sqrt{l_1 v_1^2} + \big\vert \big\vert v_h  \big\vert\big\vert_{0,\omega} \Big) \nonumber\\
&   & \nonumber\\
& \leq   &   C_{18} h\Bigg[\int_{x_0}^{x_{N+1}} \Big(\big\vert F' (\omega)\big\vert +\big\vert \omega' \big\vert + \big\vert \omega \big\vert \Big)^2dx \Bigg]^{1/2}   \Big(l_1 v_1^2 + \big\vert \big\vert v_h  \big\vert\big\vert_{0,\omega}^2 \Big)^{1/2} \nonumber\\
%&    & \nonumber\\
&  \leq  &  C_{18} \Bigg[\int_{\Omega} \Big(\big\vert F'(\omega) \big\vert +\big\vert \omega' \big\vert + \big\vert \omega \big\vert \Big)^2dx \Bigg]^{1/2}  \cdot h \cdot   \big\vert \big\vert v_h  \big\vert\big\vert_{\omega,d} \nonumber\\
%&    & \nonumber\\
 &  \leq   &  C_{211}\Bigg(\big\vert F (\omega)\big\vert_1 +\big\vert\big\vert \omega \big\vert\big\vert_{1}\Bigg)\cdot h \cdot   \big\vert \big\vert v_h  \big\vert\big\vert_{\omega,d}.
\end{eqnarray}
Note that $\Vert. \Vert_1$ and $\vert. \vert_1$ are respectively the $H^1(\Omega)$ norm  and semi-norm.

For the second term  of \eqref{deltat}, $\delta_{212}$  is estimated as in \cite{angermann2007convergence}  and we have
%it exactly as in \cite{wang}
%{\color{blue}			
%For the second term of (\ref{bil-conv}), we have 
		\begin{eqnarray*}
			\Bigg\vert c(s)\Big(\big(L_h-I\big)\omega,L_hv_h\Big) \Bigg\vert & \leq  &  C_{212}  \cdot h \cdot  \big\vert\big\vert \omega \big\vert\big\vert_{1,\omega} \cdot \big\vert\big\vert  v_h \big\vert\big\vert_{0,h}.
			\end{eqnarray*}
			Thus
			\begin{eqnarray}
			|\delta_{21}(\omega,v_h,s)|  & \leq  &  \Bigg\vert \sum_{i=0}^{N} \Bigg(F_h\big(I_h \omega(x_{i+\frac{1}{2}})\big)-F\big(\omega(x_{i+\frac{1}{2}})\big)\Bigg) \big(v_i-v_{i+1}\big) \Bigg\vert \nonumber\\
			%&    & \nonumber\\
			&    &  +  \Bigg\vert c(s)\Big(\big(L_h-I\big)\omega,L_hv_h\Big) \Bigg\vert  \nonumber\\
			%&     & \nonumber\\
			&  \leq   &     C_{211}\Bigg(\big\vert F(\omega)\big\vert_1+\big\vert\big\vert \omega \big\vert\big\vert_{1}\Bigg)\cdot h \cdot   \big\vert \big\vert v_h  \big\vert\big\vert_{\omega,d} 
			+ C_{212}  \cdot h \cdot  \big\vert\big\vert \omega \big\vert\big\vert_{1,\omega} \cdot \big\vert\big\vert  v_h \big\vert\big\vert_{0,h} \\
			%&    & \nonumber\\
			&   \leq  & C_{211} \Bigg(\big\vert F(\omega)\big\vert_{1} +\big\vert\big\vert \omega \big\vert\big\vert_{1,\omega}\Bigg) \cdot  h  \cdot \vert\vert v_h \vert\vert_{\omega,d} + C_{212} \cdot  h  \cdot  \big\vert\big\vert \omega \big\vert\big\vert_{1,\omega} \cdot  \big\vert\big\vert  v_h \big\vert\big\vert_{\omega,d} \nonumber\\
			%&    & \nonumber\\
			|\delta_{21}(\omega,v_h,s)| &  \leq   & C_{21} \Bigg(\big\vert F(\omega)\big\vert_{1} +\big\vert\big\vert \omega \big\vert\big\vert_{1,\omega}\Bigg) \cdot  h \cdot \vert\vert v_h \vert\vert_{\omega,h}
			\end{eqnarray}
			Using (\ref{Y21-exp}), we have
			\begin{eqnarray}
			\label{Y21-est}
			|Y_{21}^m| &  \leq & \theta \Big\vert \delta_{21}\Big(u(t_{m+1}),v_h,t_{m+\theta}\Big) \Big\vert +(1-\theta) \Big\vert \delta_{21}\Big(u(t_m),v_h,t_{m+\theta}\Big)\Big\vert \nonumber \\
			&    & \nonumber \\
			& \leq & C_{21} \Bigg(\big\vert F(\omega)\big\vert_{1} +\big\vert\big\vert \omega \big\vert\big\vert_{1,\omega}\Bigg) \cdot  h \cdot \vert\vert v_h \vert\vert_{\omega,h}
			\end{eqnarray}
			%}			
			%\item 
			Estimate of $Y_{22}^m$ is done as \cite{angermann2007convergence}and we have  
			\begin{equation}
			\label{Y22-est}
			|Y_{22}^m| \leq C_{22} \Delta t_m  ||v_h||_{0,h}.
			\end{equation}
		Estimate of $Y^m_3$ is done as in \cite[$Y^m_4$]{angermann2007convergence}  and we have 
        \begin{equation}
		\label{Y3-est}
		\big\vert Y_3^m \big\vert  \leq  C_3 h \big\vert\big\vert v_h \big\vert\big\vert_{0,h}.
		\end{equation}  
		%}	
	%\end{itemize}
	Coming back to the equation above 
	\begin{equation*}
	\sum_{i=1}^N  l_i \frac{W^{m+1}_i-W^m_i}{\Delta t_m}v_i+a_h\Big(\theta W^{m+1}+(1-\theta)W^m,v_h;t_{m+\theta}\Big)=Y_1^m+Y_2^m+Y_3^m.
	\end{equation*}
	Using  (\ref{Y1-est}),(\ref{Y21-est}),(\ref{Y22-est}) and (\ref{Y3-est}) we get
	\begin{eqnarray}
	\label{est1}
	\sum_{i=1}^N l_i \frac{W^{m+1}_i-W^m_i}{\Delta t_m}v_i+a_h\Big(\theta W^{m+1}+(1-\theta)W^m,v_h;t_{m+\theta}\Big)  & \leq &  \Bigg( \frac{1}{\Delta t_m} \int_{t_{m}}^{t_{m+1}} ||\big(L_h-I\big)\dot{u}(s)||_{L^2(\Omega)} ds +  \nonumber \\
	&    & \nonumber\\
	&    & + \int_{t_{m}}^{t_{m+1}} ||\ddot{u} ||_{L^2(\Omega)} ds\Bigg) \vert\vert v_h\vert\vert_{0,h}  \nonumber  \\
	&    & \nonumber\\
	&    &  +C_{21}\cdot h \cdot ||v_h||_{1,\omega}+C_{22} (\Delta t_m) ||v_h||_{0,h} \nonumber \\
	&   &  \nonumber \\
	&   & +C_3 h \big\vert\big\vert v_h \big\vert\big\vert_{0,h}.
	\end{eqnarray}
	By replacing $v_h$ by $W_h^{\theta}=\theta W^{m+1}+(1-\theta)W^m$ in (\ref{est1}), as in \cite{angermann2007convergence}  using the coercivity  property, we have:  % \ref{stab}  a chercher
	\begin{eqnarray}
	\label{ineq1}
	\frac{1}{2\Delta t_m}\Bigg[\vert\vert W^{m+1}\vert\vert_{0,h}^2 - \vert\vert W^m\vert\vert_{0,h}^2\Bigg] + \alpha  \vert\vert W_h^{\theta}\vert\vert_{\omega,d}^2 & \leq  & \mathcal{Q}^m(\Delta t_m,h) ||W_h^{\theta}||_{\omega,d}\nonumber\\
	\end{eqnarray}
	with 
	\begin{equation*}
	\mathcal{Q}^m(\Delta t_m,h) \leq \mathcal{Q}_1^m(\Delta t_m,h)+C'\cdot \Big(h+\Delta t_m \Big),
	\end{equation*}
	and
%{\color{blue}	
	\begin{equation*}
	\mathcal{Q}_1^m(\Delta t_m,h)=\frac{1}{\Delta t_m} \int_{t_{m}}^{t_{m+1}} ||\big(L_h-I\big)\dot{u}(s)||_{L^2(\Omega)} ds +\int_{t_{m}}^{t_{m+1}} ||\ddot{u} ||_{L^2(\Omega)} ds.
	\end{equation*}
  Following \cite{angermann2007convergence}, we therefore
	\begin{eqnarray}
	\sum_{k=0}^{m-1} \Delta t_k \Big[\mathcal{Q}_1^k\big(\Delta t_k,h\big)\Big] & \leq & 2 \sum_{k=0}^{m-1}\Bigg[\int_{t_{k+1}}^{t_k}\vert\vert(L_h-I)\dot{u}(s)\vert\vert_{L^2(\Omega)}^2 ds + (\Delta t_k)^2 \int_{t_{k+1}}^{t_k}\vert\vert\ddot{u}(s)\vert\vert_{L^2(\Omega)}^2 ds\Bigg] \nonumber\\
	%&    & \nonumber\\
	& \leq & 2 \int_0^T \vert\vert(L_h-I)\dot{u}(s)\vert\vert_{L^2(\Omega)}^2 ds + (\Delta t)^2 \int_0^T \vert\vert\ddot{u}(s)\vert\vert_{L^2(\Omega)}^2 ds,
	\end{eqnarray}
	where $\Delta t:=\max_{m=0,\ldots,M-1} \vert \Delta t_m \vert$. As in  \cite{angermann2007convergence}, we have 
	\begin{equation*}
	\vert\vert (L_h-I)\dot{u}(s) \vert\vert_{L^2(\Omega)} \leq h \vert \dot{u}(s) \vert_1
	\end{equation*}
	then we get
	\begin{equation}
	\label{SQ-1-est}
	\sum_{k=0}^{m-1} \Delta t_k \Big[\mathcal{Q}_1^k \big(\Delta t_k,h\big)\Big]  \leq  2\Bigg[ h^2 \vert\vert \dot{u} \vert\vert^2_{L^2\big(0,T;H^1(\Omega)\big)}+(\Delta t)^2 \vert\vert \ddot{u} \vert\vert^2_{L^2\big(0,T;L^2(\Omega)\big)} \Bigg].
	\end{equation}
	%Using (\ref{Q-est}),(\ref{SQ-1-est}) and (\ref{t-est}) , we get
	Following \cite[(66)]{angermann2007convergence}, (\ref{SQ-1-est})  yields 
	\begin{equation}
	\vert\vert W^m \vert\vert_{0,h}^2 \leq \vert\vert W^0 \vert\vert_{0,h}^2 + C'\Big(h^2+(\Delta t)^2\Big).
	\end{equation}
	By taking $u^0=I_hu_0$,  we have  $\vert\vert W^0 \vert\vert_{0,h} =0$ and  therefore
	\begin{equation}
	\vert\vert W^m \vert\vert_{0,h}  \leq C(h+\Delta t),
	\end{equation}
	which is actually
	\begin{equation}
	\label{proj-err}
	\vert\vert I_hu(t_m) - u_h^m \vert\vert_{0,h} \leq C(h+\Delta t).
	\end{equation}
	Therefore, using \eqref{interp-err} and \eqref{proj-err}, we get
	\begin{eqnarray}
	\big\vert\big\vert u(t_m)-u_h^m\big\vert\big\vert_{0,h} & \leq & \big\vert\big\vert u(t_m)-I_hu(t_m)\big\vert\big\vert_{0,h} + \big\vert\big\vert I_hu(t_m)-u_h^m\big\vert\big\vert_{0,h} \nonumber\\
	%&   & \nonumber\\
	& \leq  &  C_{32} \cdot h + C'(h+\Delta t) \nonumber\\
	%&     & \nonumber\\
	\big\vert\big\vert u(t_m)-u_h^m\big\vert\big\vert_{0,h} & \leq    &    C(h+\Delta t).
	\end{eqnarray}
	%\item  
\underline{$2^{nd} case$:} The  proof for  fitted TPFA method is done exactly in the same way.

%Here the numerical solution $\zeta_h^m$ is replaced by $z_h^m$  and the bilinear form $a_h$ in  \ref{est-A} and \ref{A-Y2}, is replaced by bilinear form $b_h$ defined in \ref{bilbflux-fitTPFA}. Thereby, $\delta_{21}$ in \ref{Y2-delta1} is given by:
%
%\begin{equation}
%\delta_{21}(\omega,v_h,s)      =   \sum_{i=0}^{N} \Bigg(G_h\big(I_h \omega(x_{i+\frac{1}{2}})\big)-F\big(\omega(x_{i+\frac{1}{2}})\big)\Bigg) \cdot \big(v_i-v_{i+1}\big) + b(s)\Big(\big(L_h-I\big)\omega,L_hv_h\Big) \nonumber\\
%\end{equation}
%
%where $G_h$ is defined in 	\ref{disc-flux-fitTPFA}. Then we get
%
%\begin{equation}
%|\delta_{21}(\omega,v_h,s)|   \leq    C_{21} \Bigg(\big\vert\big\vert F'(\omega) \big\vert\big\vert_{\mathbf{L^2(\Omega)}} +\big\vert\big\vert \omega \big\vert\big\vert_{1,\omega}\Bigg) \cdot  h \cdot \vert\vert v_h \vert\vert_{\omega,d}
%\end{equation}	
%
%Therefore, following similar estimates of $Y_1$, $Y_2$ and $Y_3$  for the TPFA method case, we get: 
%
%\begin{equation}
%	\big\vert\big\vert u(t_m)-z_h^m\big\vert\big\vert_{0,h}  \leq       C(h+\Delta t)
%\end{equation}
%%\end{itemize}	
	\end{proof}

\section{Numerical experiments}
\label{num-exp}
In this section, we perform numerical experiments for an European call option pricing problem. The error are computed with respect to the following analytical solution of the Black-Scholes PDE (see \cite{haug2007complete}): 
\begin{equation}
\label{analytical}
C(x,t)= xN(d_1)-Ke^{-rt}N(d_2),
\end{equation}
where
\begin{equation}
d_1=\dfrac{\ln(\frac{x}{K})+\big(r+\frac{\sigma^2}{2}\big)}{t}~~~~~~~~d_2=d_1-\sigma \sqrt{t}.
\end{equation}
with $t$ the time to maturity and $N(\cdot)$  the standard cumulative normal distribution function.
 The computational domain is $\Omega=[0,x_{\max}]\times (0,T]$ with $x_{\max}=300$ and the maturity time $T=1$. 
 These numerical experiments are performed using the risk free interest  rate $r=0.1$, the volatilty $\sigma=0.5$ and the strike price $K=100$.
\begin{figure}[htbp]
	\centering
	(a)
	\includegraphics[scale=0.4]{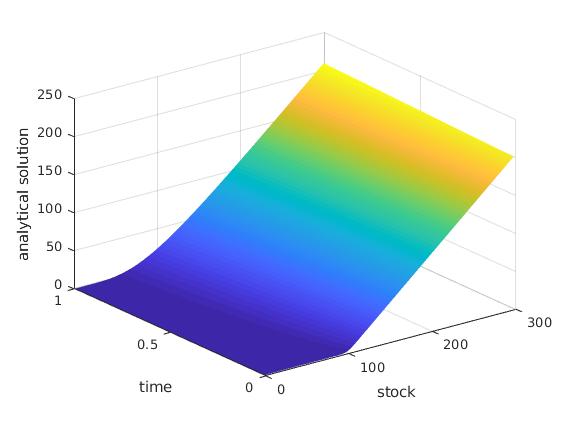}
	(b)
	\includegraphics[scale=0.4]{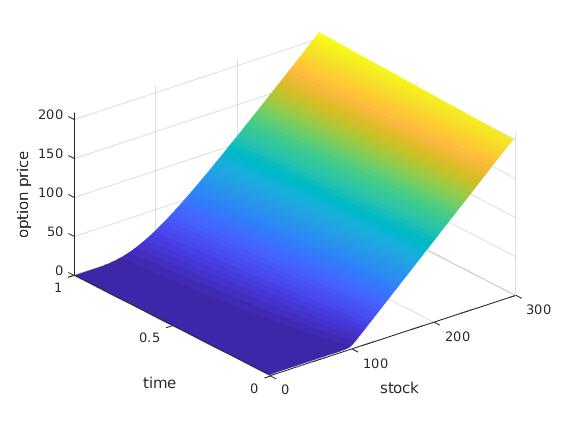}
	\caption{Analytical solution in (a) and numerical solution from fitted TPFA in (b) at maturity time $T=1$}
\end{figure}
%\newpage
%%and the numerical solution obtained using the fitted TPFA coupled to the upwind method gives
%\begin{figure}[htbp]
%	\centering
%	\includegraphics[scale=0.4]{TPFA}
%	\caption{Numerical solution using fitted TPFA-upwind method}
%\end{figure}
%\newpage 
%\paragraph{Space and time discretisation Error}
%Here, we evaluate the error using the fitted TPFA method with respect to the analytical solution \ref{analytical}. 
Here  we have  performed space  errors by  fixing the time step at $dt=1/100$ and  vary the space step $h$. 
\begin{table}[!h]
	\centering
	\begin{tabular}{|c||c|c|c|c|c|c|c|c|c|}
		\hline
		\backslashbox{Num\\ meth}{Nb\\Grids pts}
		& 100 & 150 & 200 & 250 & 300 & 350 & 400 & 450 & 500  \\  
		%		\hline
		%		space step $h$ &  3 & 2 & 1.5 & 1.2 & 1 & 0.86 & 0.75   &	0.67  & 0.6 \\
		\hline 
		TPFA & 0.0104 &   0.0069& 0.0052 & 0.0042 &  0.0035  & 0.003& 0.0026 & 0.0023 & 0.0021 \\
		\hline
		Fitted TPFA & 0.0103 &   0.0069& 0.0052 & 0.0041 &  0.0034  & 0.0029& 0.0026 & 0.0023 & 0.0021\\
		\hline
	\end{tabular} 
	\caption{Table of  space errors. The time step is fixed to be $dt=1/100$. }
\end{table}

For the time error, we fix the space step at $h=0.25$, and vary the time step $dt$. 
%We get the following table:
\begin{table}[!h]
	\centering
	\begin{tabular}{|c||c|c|c|c|c|c|c|c|c|}
		\hline
		\backslashbox{Num\\ meth}{Nb\\Time subdivisions }
		& 100 & 150 & 200 & 250 & 300 & 350 & 400 \\  
		%		\hline
		%		time step $h$ &  3 & 2 & 1.5 & 1.2 & 1 & 0.86 & 0.75   &	0.67  & 0.6 \\
		\hline 
		TPFA & $8.98.10^{-4}$ & $8.83.10^{-4}$& $8.75.10^{-4}$ & $8.71.10^{-4}$ &  $8.69.10^{-4}$ & $8.66.10^{-4}$& $8.656.10^{-4}$ \\
		\hline
		Fitted TPFA &$8.98.10^{-4}$ &   $8.83.10^{-4}$& $8.75.10^{-4}$ & $8.71.10^{-4}$ &  $8.69.10^{-4}$  &$8.66.10^{-4}$& $8.656.10^{-4}$ \\
		\hline
	\end{tabular}
	\caption{Table of  time  errors.  The  space  step is fixed to be $h=0.25$. }
\end{table}
%\newpage 
%We therefore get the following graph error
\begin{figure}[htbp]
	\centering
	\includegraphics[scale=0.4]{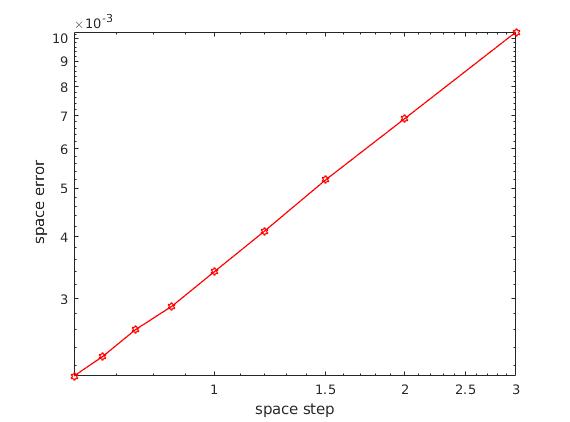}
	\caption{  Space step versus  $L^2$ Errors  in loglog scale. This graph  shows  the convergence in space of the  fitted TPFA.  The order of convergence in space is $\mathcal{O}(h)$,  this is in agreement with the theoretical result in Theorem \ref{maintheo}.
	 The time step is fixed to be $dt=1/100$.}
\end{figure}
\section{Conclusion}
In this paper, we have presented  two spatial numerical methods for spatial  discretization  of the Black-Scholes PDE for pricing options. 
The first scheme is the classical finite volume method with Two-Point Flux Approximation (TPFA) and the second scheme  is  a novel scheme called fitted Two-Point Flux Approximation (TPFA).
The novel fitted Two-Point Flux Approximation (TPFA)  combines the classical fitted finite volume method   and the  standard TPFA method.  
The  classical  fitted finite volume method is used to tackle the degeneracy of the   Black-Scholes  PDE near zero. 
The convergence analyses are performed  along with numerical experiments to confirm  the theoretical results.

\section*{Acknowledgement}
This work was supported by the Robert Bosch Stiftung through the AIMS ARETE Chair programme (Grant No 11.5.8040.0033.0).

%\newpage
%\bibliography{mybibfile}

\section*{References}
\begin{thebibliography}{10}
\bibitem{black1973pricing}
F.~Black, M.~Scholes, The pricing of options and corporate liabilities.
 Journal  of political economy 81~(3) (1973) 637--654.

\bibitem{cox1979option}
J.~C. Cox, S.~A. Ross, M.~Rubinstein, Option pricing: A simplified approach.
  Journal of financial Economics 7~(3) (1979) 229--263.

\bibitem{duffy2013finite}
D.~J. Duffy, Finite Difference methods in financial engineering: a Partial
  Differential Equation approach. John Wiley \& Sons, 2013.

\bibitem{topper2005financial}
J.~Topper, Financial engineering with finite elements. Vol. 319, John Wiley \&
  Sons, 2005.

\bibitem{wang2004novel}
S.~Wang, A novel fitted finite volume method for the Black--Scholes equation
  governing option pricing. IMA Journal of Numerical Analysis  24~(4) (2004)
  699--720.

\bibitem{wang2006power}
S.~Wang, X.~Yang, K.~Teo, Power penalty method for a linear complementarity
  problem arising from american option valuation. Journal of optimization theory 
  and applications 129~(2) (2006) 227--254.

\bibitem{angermann2007convergence}
L.~Angermann, S.~Wang, Convergence of a fitted finite volume method for the
  penalized Black--Scholes equation governing European and American Option
  pricing.  Numerische Mathematik 106~(1) (2007) 1--40.

\bibitem{koffi2019fitted}
R.~S. Koffi, A.~Tambue, A Fitted Multi-point Flux Approximation Method for Pricing Two Options. 
Computational Economics, 2019. https://doi.org/10.1007/s10614-019-09906-x.

\bibitem{eymard2000finite}
R.~Eymard, T.~Gallou{\"e}t, R.~Herbin, Finite volume methods.  Handbook of
  numerical analysis 7 (2000) 713--1018.

\bibitem{tambue2016exponential}
A.~Tambue, An exponential integrator for finite volume discretization of a
  reaction--advection--diffusion equation.  Computers \& Mathematics with
  Applications 71~(9) (2016) 1875--1897.

%\bibitem{aavatsmark2007multipoint}
%I.~Aavatsmark, Multipoint flux approximation methods for quadrilateral grids,
%  in: 9th International forum on reservoir simulation, Abu Dhabi, 2007, pp.
%  9--13.
\bibitem{KnaberAngermann2002}
L.~A. P.~Knabner, Numerical Methods for Elliptic and Parabolic Partial
  Differential Equations, Vol.~44, Springer, 2002.
 \bibitem{existence}
 J. Haslinger, M. Miettinen, and P. D. Panagiotopoulos,  Finite Element Method for Hemivariational Inequalities.
   Dordrecht, The Netherlands: Kluwer, 1999.
   \bibitem{haug2007complete}
Haug, Espen Gaarder, The complete guide to option pricing formulas. Vol 2, McGraw-Hill New York, 2007.
\end{thebibliography}

\end{document}